\documentclass[a4paper,reqno,10pt]{amsart}

\DeclareRobustCommand{\SkipTocEntry}[5]{}

\usepackage{hyperref}
\usepackage{a4wide}

\usepackage[justification=centering]{caption}
\usepackage{subcaption}

\usepackage{amssymb}
\usepackage{amstext}
\usepackage{amsmath}
\usepackage{amscd}
\usepackage{amsthm}
\usepackage{amsfonts}

\usepackage{graphicx}
\usepackage{latexsym}
\usepackage{mathrsfs}

\usepackage{enumitem}
\setlist[enumerate,1]{label={\upshape(\arabic*)}}
\setlist[enumerate,2]{label={\upshape(\alph*)}}

\usepackage{tikz}
\usetikzlibrary{cd,arrows,matrix,backgrounds,shapes,positioning,calc,decorations.markings,decorations.pathmorphing,decorations.pathreplacing}
\tikzset{blackv/.style={circle,fill=black,inner sep=3pt,outer sep=3pt},
         whitev/.style={circle,fill=white,draw=black,inner sep=3pt,outer sep=3pt},
         blabel/.style={
           circle, fill=white, draw=black, font=\scriptsize,
           inner sep=0.5pt, outer sep=0pt},
         redv/.style={circle,fill=red,inner sep=3pt,outer sep=3pt},
         block/.style={draw,rectangle split,rectangle split horizontal,rectangle split parts=#1},
         symbol/.style={
           draw=none,
           every to/.append style={
             edge node={node [sloped, allow upside down, auto=false]{$#1$}}}}
}

\usepackage{rotating}

\usepackage{array}
\newcolumntype{C}{>{$}c<{$}}

\usepackage{makecell}

\newtheorem{theorem}{Theorem}[section]
\newtheorem{theoremi}{Theorem}
\newtheorem{corollaryi}[theoremi]{Corollary}
\newtheorem{propositioni}[theoremi]{Proposition}

\newtheorem{corollary}[theorem]{Corollary}
\newtheorem{lemma}[theorem]{Lemma}
\newtheorem*{lemma*}{Lemma}
\newtheorem*{theorem*}{Theorem}
\newtheorem{proposition}[theorem]{Proposition}
\newtheorem{definition-proposition}[theorem]{Definition-Proposition}

\newtheorem{question}[theorem]{Question}
\newtheorem{conjecture}[theorem]{Conjecture}

\theoremstyle{definition}
\newtheorem{definition}[theorem]{Definition}

\newtheorem{remark}[theorem]{Remark}
\newtheorem{example}[theorem]{Example}

\newcommand{\2}{\mathbf{2}}

\newcommand{\CC}{\mathcal{C}}

\newcommand{\EE}{\mathcal{E}}

\newcommand{\II}{\mathcal{I}}

\newcommand{\PP}{\mathcal{P}}

\newcommand{\XX}{\mathcal{X}}

\newcommand{\YY}{\mathcal{Y}}

\newcommand{\Ext}{\operatorname{Ext}\nolimits}
\newcommand{\Hom}{\operatorname{Hom}\nolimits}
\newcommand{\End}{\operatorname{End}\nolimits}
\newcommand{\op}{\operatorname{op}\nolimits}
\newcommand{\rad}{\operatorname{rad}\nolimits}

\newcommand{\Image}{\operatorname{Im}\nolimits}
\newcommand{\Kernel}{\operatorname{Ker}\nolimits}

\newcommand{\im}{\Image}
\renewcommand{\ker}{\Kernel}

\DeclareMathOperator{\moduleCategory}{\mathsf{mod}} \renewcommand{\mod}{\moduleCategory}

\DeclareMathOperator{\proj}{\mathsf{proj}}

\DeclareMathOperator{\inj}{\mathsf{inj}}
\DeclareMathOperator{\ind}{\mathsf{ind}}

\DeclareMathOperator{\GP}{\mathsf{GP}}

\DeclareMathOperator{\wtilt}{\mathsf{W-tilt}}
\DeclareMathOperator{\tilt}{\mathsf{tilt}}

\DeclareMathOperator{\Fac}{\mathsf{Fac}}

\DeclareMathOperator{\add}{\mathsf{add}}

\DeclareMathOperator{\pd}{\mathsf{pd}}
\DeclareMathOperator{\id}{\mathsf{id}}

\newcommand{\iso}{\cong}

\newcommand{\defl}{\twoheadrightarrow}

\newcommand{\equi}{\simeq}
\newcommand{\sst}[1]{\substack{#1}}

\numberwithin{equation}{section}

\begin{document}
\title{Maximal self-orthogonal modules and a new generalization of tilting modules}

\author[H. Enomoto]{Haruhisa Enomoto}
\address{Graduate School of Science, Osaka Metropolitan University, 1-1 Gakuen-cho, Naka-ku, Sakai, Osaka 599-8531, Japan}
\email{henomoto@omu.ac.jp}

\begin{abstract}
  We introduce a generalization of tilting modules of finite projective dimension, projectively Wakamatsu tilting modules, which are self-orthogonal and Ext-progenerators in their Ext-perpendicular categories. Under a certain finiteness condition, we prove that the following modules coincide: projectively Wakamatsu tilting, Wakamatsu tilting, maximal self-orthogonal, and self-orthogonal modules with the same rank as the algebra. This provides another proof of the weak Gorensteinness of representation-finite algebras.
  To prove this, we introduce Bongartz completion of self-orthogonal modules and characterize its existence.
  Moreover, we study a binary relation on Wakamatsu tilting modules which extends the poset of tilting modules, and use it to prove that every self-orthogonal module over a representation-finite Iwanaga-Gorenstein algebra has finite projective dimension.
  Finally, we discuss several conjectures related to self-orthogonal modules and their connections to famous homological conjectures.
\end{abstract}

\maketitle

\tableofcontents

\section{Introduction}
In the representation theory of artin algebras, \emph{self-orthogonal modules} have been widely studied. A $\Lambda$-module $T$ is self-orthogonal if $\Ext^i_\Lambda(T, T) = 0$ for all $i > 0$. For example, the famous \emph{Auslander-Reiten conjecture} states that a self-orthogonal generator must be projective. One of the most extensively studied classes of self-orthogonal modules is \emph{tilting modules} of finite projective dimension introduced by Miyashita \cite{miyashita}. Tilting modules induce a derived equivalence and can be used to classify a particular class of coresolving subcategories of $\mod\Lambda$ \cite{applications}, forming the basis of what is known as \emph{tilting theory}.

For a self-orthogonal module $T$, consider the subcategory $T^\perp$ of $\mod\Lambda$ consisting of $X$ such that $\Ext_\Lambda^i(T, X) = 0$ for all $i > 0$. If $T$ is a tilting module, then $T$ becomes an $\Ext$-progenerator of $T^\perp$ \cite{applications}.
Taking this into account, we introduce a new generalization of a tilting module, called a \emph{projectively Wakamatsu tilting module}:
\begin{definition}
  Let $T$ be a self-orthogonal $\Lambda$-module. We call $T$ \emph{projectively Wakamatsu tilting} if $T$ is an $\Ext$-progenerator of $T^\perp$.
\end{definition}
Then a tilting module is precisely a projectively Wakamatsu tilting module of finite projective dimension (Proposition \ref{prop:tilt-rels}).
This paper aims to study projectively Wakamatsu tilting modules, their relationship to Wakamatsu tilting modules, and to apply these results to the study of self-orthogonal modules.

The name \emph{projectively Wakamatsu tilting} is derived from the fact that this class is a subclass of \emph{Wakamatsu tilting modules} \cite{wakamatsu}. We recall their definition briefly.
Auslander-Reiten \cite{applications} introduced the category $\YY_T \subseteq T^\perp$ for any self-orthogonal module $T$, such that $T$ is an $\Ext$-progenerator of $\YY_T$ (Definition \ref{def:subcat}). Then a module $T$ is called \emph{Wakamatsu tilting} if $D \Lambda \in \YY_T$. Wakamatsu tilting modules are of particular interest when considering exact categories, since the result of \cite{eno-wak} implies that if an exact category $\EE$ has a progenerator $P$ and an injective cogenerator $I$, then $I$ is a Wakamatsu tilting module under the canonical embedding $\EE(P,-) \colon \EE \hookrightarrow \mod \End_\EE(P)$.

We prove that a projectively Wakamatsu tilting module is precisely a Wakamatsu tilting module $T$ for which $\YY_T = T^\perp$ holds (Proposition \ref{prop:wak-proj-char}).
One of the drawbacks of Wakamatsu tilting modules is that the category $\YY_T$ is unclear compared to $T^\perp$. Thus, projectively Wakamatsu tilting modules can be seen as a suitable subclass of Wakamatsu tilting modules:
\[
  \{ \text{tilting} \}
  \quad \subseteq \quad
  \{ \text{projectively Wakamatsu tilting}  \}
  \quad \subseteq \quad
  \{ \text{Wakamatsu tilting} \}
\]

The main result of this study, which focuses on representation-finite algebras, is as follows:
\begin{theoremi}[= Theorem \ref{thm:main2}]\label{thm:A}
  Let $\Lambda$ be an artin algebra and $T \in \mod\Lambda$. Suppose that $T^\perp$ has only finitely many indecomposables (e.g. $\Lambda$ is representation-finite). Then the following conditions are equivalent:
  \begin{enumerate}
    \item $T$ is a projectively Wakamatsu tilting module.
    \item $T$ is a Wakamatsu tilting module.
    \item $T$ is a maximal self-orthogonal module, that is, $T$ is self-orthogonal, and if $T \oplus M$ is self-orthogonal, then $M \in \add T$ holds.
    \item $T$ is a self-orthogonal module with $|T| = |\Lambda|$.
  \end{enumerate}
\end{theoremi}
This theorem has several applications. Firstly, this enables us to easily find projectively Wakamatsu tilting = Wakamatsu tilting modules for the representation-finite case, since the conditions (3) and (4) can be easily checked for any given algebra and module, and a computer program can be used to obtain all such modules. This was in fact the original motivation for this paper (see Section \ref{sec:example} for various concrete examples).

Secondly, it provides a unified proof of some homological conjectures for representation-finite algebras. Let $\Lambda$ be a representation-finite artin algebra. Since $\Lambda_\Lambda$ satisfies (4) in Theorem \ref{thm:A}, it satisfies (3), so $\Lambda$ is maximal self-orthogonal, which is precisely the \emph{Auslander-Reiten conjecture}.
Moreover, let $T$ be a Wakamatsu tilting module of finite projective dimension. Then $T$ is projectively Wakamatsu tilting by Theorem \ref{thm:A}, so $T$ is tilting by Proposition \ref{prop:tilt-rels}, which is precisely the \emph{Wakamatsu tilting conjecture}. One can also deduce the Generalized Nakayama conjecture easily (see Proposition \ref{prop:conj-rel}).

Finally, Theorem \ref{thm:A} immediately gives another proof of the following result in Gorenstein homological algebra, which was previously shown in \cite[Corollary 5.11]{be}.
Here $\GP \Lambda$ denotes the category of Gorenstein-projective modules.
\begin{corollary}[= Corllary \ref{cor:weakly-gor}]
  Let $\Lambda$ be an artin algebra such that $^\perp \Lambda$ has only finitely many indecomposables. Then $\Lambda$ is weakly Gorenstein, that is, $\GP \Lambda = {}^\perp \Lambda$ holds.
\end{corollary}

The main tool for proving Theorem \ref{thm:A} is \emph{Bongartz completion} of a self-orthogonal module. A Bongartz completion of a self-orthogonal module $U$ is a projectively Wakamatsu tilting module $T$ such that $T^\perp = U^\perp$, which generalizes Bongartz completion of tilting modules.
We prove that $U$ has a Bongartz completion if and only if $U^\perp$ has a finite cover (Theorem \ref{thm:main}). As a consequence, every self-orthogonal module over a representation-finite algebra can be completed into a projectively Wakamatsu tilting module.

We also investigate the following binary relation $\leq$ on the set of Wakamatsu tilting $\Lambda$-modules. Let $\wtilt\Lambda$ be the set of isomorphism classes of basic Wakamatsu tilting $\Lambda$-modules. For $T_1, T_2 \in \wtilt\Lambda$, we define $T_1 \geq T_2$ if $\Ext^{>0}_\Lambda(T_1, T_2) = 0$. This extends and contains the known poset of (co)tilting modules.
We find however, that \emph{$(\wtilt\Lambda, \leq)$ is not a poset in general} even if $\Lambda$ is representation-finite (Example \ref{ex:not-poset}), which is in contrast to the poset of tilting modules.
As for this binary relation, we prove that the set of tilting modules is upward-closed in $\wtilt\Lambda$:
\begin{theoremi}[= Theorem \ref{thm:tilt-up-closed}]
  Let $\Lambda$ be a representation-finite artin algebra, and $T, U \in \wtilt\Lambda$. If $T \geq U$ holds in $\wtilt\Lambda$ and $U$ is tilting, then so is $T$.
\end{theoremi}
This theorem has the following consequence on self-orthogonal modules over Iwanaga-Gorenstein algebras, which is of particular interest in itself.
\begin{corollaryi}[= Corollary \ref{cor:IG}]
  Let $\Lambda$ be a representation-finite Iwanaga-Gorenstein algebra. Then the following hold.
  \begin{enumerate}
    \item Every self-orthogonal $\Lambda$-module has finite projective dimension.
    \item Tilting modules, projectively Wakamatsu tilting modules, Wakamatsu tilting modules, and cotilting modules are all equivalent.
  \end{enumerate}
\end{corollaryi}

If $\Lambda$ is not representation-finite, then some conditions in Theorem \ref{thm:A} are not equivalent (see Examples \ref{ex:weakly-gor} and \ref{ex:mso-not-wak}). However, we conjecture that some of them are still related:
\begin{conjecture}\label{conj:intro}
  Every Wakamatsu tilting module $T$ is maximal self-orthogonal with $|T| = |\Lambda|$. Furthermore, every self-orthogonal module $T$ with $|T| = |\Lambda|$ is maximal self-orthogonal and Wakamatsu tilting.
\end{conjecture}
In Section \ref{sec:conj}, we propose several conjectures related to this and its relation to famous homological conjectures such as the Auslander-Reiten conjecture and the Generalized Nakayama conjecture.
Here, we only mention one of the results:
\begin{propositioni}[= Proposition \ref{prop:conj-rel} (5)]
  The following conjectures are equivalent:
  \begin{enumerate}
    \item The Auslander-Reiten conjecture: $\Lambda_\Lambda$ is maximal self-orthogonal.
    \item Weak maximal self-orthogonal conjecture: If $T$ is projectively Wakamatsu tilting, then $T$ is maximal self-orthogonal.
  \end{enumerate}
\end{propositioni}

\addtocontents{toc}{\SkipTocEntry}
\subsection*{Conventions and notation}
Throughout this paper, we denote by $\Lambda$ an artin $R$-algebra over a commutative artinian ring $R$.
We often omit the base ring $R$ and simply refer to $\Lambda$ as an artin algebra.  The category of finitely generated right $\Lambda$-modules is denoted by $\mod\Lambda$. The subcategory of $\mod\Lambda$ consisting of all projective (resp. injective) $\Lambda$-modules is denoted by $\proj\Lambda$ (resp. $\inj\Lambda$).
The Matlis duality functor is denoted by $D: \mod\Lambda \to \mod\Lambda^{\op}$.

For modules $X, Y \in \mod\Lambda$ and $d \geq 0$, we use the notation $\Ext_\Lambda^{>d}(X, Y) = 0$ to indicate that $\Ext_\Lambda^i(X, Y) = 0$ for all $i > d$. We use similar notation for subcategories $\CC$ of $\mod\Lambda$ such as $\Ext_\Lambda^{>0}(\CC, X) = 0$.

All subcategories are assumed to be full and closed under isomorphisms.
For a $\Lambda$-module $X \in \mod\Lambda$, we denote by $\add X$ the subcategory of $\mod\Lambda$ consisting of all direct summands of finite direct sums of $X$. The number of non-isomorphic indecomposable direct summands of $X$ is denoted by $|X|$.
For a subcategory $\CC$ of $\mod\Lambda$ closed under direct summands, we denote by $\ind\CC$ the set of isomorphism classes of indecomposable objects in $\CC$, and $\#\ind\CC$ denotes its cardinality.

Let $\EE$ be a subcategory of $\mod\Lambda$ closed under extensions and direct summands. We denote by $\PP(\EE)$ the subcategory of $\EE$ consisting of objects which are \emph{$\Ext$-projective} in $\EE$, that is, objects $P \in \EE$ such that $\Ext^1_\Lambda(P, \EE) = 0$.
Note that we only consider $\Ext^1$. Dually, we denote by $\II(\EE)$ the subcategory of $\EE$ consisting of objects in $\EE$ which are $\Ext$-injective, that is, objects $I \in \EE$ such that $\Ext^1_\Lambda(\EE, I) = 0$.
We say that $\EE$ has \emph{enough $\Ext$-projectives} if for every $X \in \EE$, there exists a short exact sequence
\[
  \begin{tikzcd}
    0 \rar & X' \rar & P_0 \rar & X \rar & 0
  \end{tikzcd}
\]
in $\mod\Lambda$ with $X' \in \EE$ and $P_0 \in \PP(\EE)$. If $\EE$ has enough $\Ext$-projectives and $P \in \EE$ satisfies $\add P = \PP(\EE)$, then $P$ is called an \emph{$\Ext$-progenerator}. This is equivalent to that for every $X \in \EE$ there exists a short exact sequence of the above form with $P_0 \in \add P$ and $X' \in \EE$.
Dually, we define \emph{enough $\Ext$-injectives} and an \emph{$\Ext$-injective cogenerator} of $\EE$.

For an exact category $\EE$, we refer to \emph{progenerators} and \emph{injective cogenerators} in the same way as $\Ext$-progenerators and $\Ext$-injective cogenerators above. If $\EE$ has enough projectives and enough injectives, then we define the higher Ext group $\Ext_\EE^i(-,-)$ on $\EE$ by using projective or injective resolutions in $\EE$ as usual. For $X \in \EE$, its projective dimension $\pd_\EE X$ and injective dimension $\id_\EE X$ are defined in the same way as for usual modules.

\section{Preliminaries}\label{sec:2}
In this section, we recall the basic concepts and results of Wakamatsu tilting theory and finite covers of a category.
We begin by recalling the concept of self-orthogonal modules, which is the main focus of this paper.
\begin{definition}
  A $\Lambda$-module $T \in \mod\Lambda$ is \emph{self-orthogonal} if $\Ext^{>0}_\Lambda(T, T) = 0$ holds.
\end{definition}

Among self-orthogonal modules, (co)tilting modules have been the most studied. We will now recall their definition.
\begin{definition}
  A $\Lambda$-module $T \in \mod\Lambda$ is \emph{tilting} if it satisfies the following conditions.
  \begin{enumerate}
    \item $\pd T_\Lambda$ is finite.
    \item $T$ is self-orthogonal.
    \item There is an exact sequence
          \[
            \begin{tikzcd}
              0 \rar & \Lambda \rar & T^0 \rar & T^1 \rar & \cdots  \rar & T^d \rar & 0
            \end{tikzcd}
          \]
          with $T^i \in \add T$.
  \end{enumerate}
  Dually, $T\in \mod\Lambda$ is \emph{cotilting} if $DT$ is tilting.
\end{definition}

For a self-orthogonal module $T$, we consider the following four subcategories of $\mod\Lambda$:
\begin{definition}\label{def:subcat}
  Let $T$ be a self-orthogonal $\Lambda$-module.
  \begin{enumerate}
    \item $^\perp T$ is the subcategory of $\mod\Lambda$ consisting of all modules $X$ such that $\Ext^{>0}_\Lambda(X, T) = 0$.
    \item $\XX_T$ is the subcategory of $\mod\Lambda$ consisting of all modules $X$ such that there exists an exact sequence
          \[
            \begin{tikzcd}
              0 \rar & X \rar["f^0"] & T^0 \rar["f^1"] & T^1 \rar["f^2"] & T^2 \rar & \cdots
            \end{tikzcd}
          \]
          with $T^i \in \add T$ and $\im f^i \in {}^\perp T$ for all $i \geq 0$ (in particular, $X \in {}^\perp T$).
          We call such a sequence a \emph{$T$-coresolution of $X$}.
    \item $T^\perp$ is the subcategory of $\mod\Lambda$ consisting of all modules $X$ such that $\Ext_\Lambda^{>0}(T, X) = 0$.
    \item $\YY_T$ is the subcategory of $\mod\Lambda$ consisting of all modules $Y$ such that there exists an exact sequence
          \[
            \begin{tikzcd}
              \cdots \rar & T_2 \rar["g_2"] & T_1 \rar["g_1"] & T_0 \rar["g_0"] & Y \rar & 0
            \end{tikzcd}
          \]
          with $T_i \in \add T$ and $\im g_i \in T^\perp$ for all $i \geq 0$ (in particular, $X \in T^\perp$).
          We call such a sequence a \emph{$T$-resolution of $Y$}.
  \end{enumerate}
\end{definition}
The subcategories ${}^\perp T$ and $\XX_T$ are defined such that $T$ behaves like an injective module in them. On the other hand, $T^\perp$ and $\YY_T$ are subcategories in which $T$ behaves like a projective module. Let us describe the basic properties of these subcategories.
Recall that a subcategory $\CC$ of $\mod\Lambda$ is called \emph{resolving} if $\Lambda \in \CC$ and $\CC$ is closed under extensions, direct summands, and kernels of surjections. Dually, $\CC$ is \emph{coresolving} if $D\Lambda \in \CC$ and $\CC$ is closed under extensions, direct summands, and cokernels of injections.
\begin{proposition}[{\cite[Proposition 5.1]{applications}}]\label{prop:4cats-projinj}
  Let $T$ be a self-orthogonal module. Then $^\perp T$, $\XX_T$, $T^\perp$, and $\YY_T$ are all closed under extensions and direct summands. In addition, they satisfy the following properties.
  \begin{enumerate}
    \item $^\perp T$ is a resolving subcategory with an $\Ext$-progenerator $\Lambda$, and $T \in \II({}^\perp T)$ holds.
    \item $\XX_T$ has an $\Ext$-injective cogenerator $T$, and is closed under extensions, direct summands, and kernels of surjections.
    \item $T^\perp$ is a coresolving subcategory with an $\Ext$-injective cogenerator $D\Lambda$, and $T \in \PP(T^\perp)$ holds.
    \item $\YY_T$ has an $\Ext$-progenerator $T$, and is closed under extensions, direct summands, and cokernels of injections.
  \end{enumerate}
\end{proposition}
The following example relates our theory to what is known as Gorenstein homological algebra.
\begin{example}
  Trivially $\Lambda_\Lambda$ itself is self-orthogonal. In this case, $\XX_\Lambda$ is precisely the category $\GP \Lambda$ of \emph{Gorenstein-projective} modules.
  On the other hand, a module in $^\perp \Lambda$ is recently called \emph{semi-Gorenstein-projective}, and its relation to Gorenstein-projective modules has been recently studied by \cite{RZ, ma2}.
\end{example}

\begin{example}\label{ex:tilt-wak-proj}
  Let $T$ be a tilting module. Then $T^\perp = \YY_T$ holds by the result of Auslander and Reiten \cite[Theorem 5.4]{applications}. Additionally, if $T$ is a classical tilting module (that is, $\pd T_\Lambda \leq 1$), then it is well-known that $T^\perp = \Fac T$, where $\Fac T$ consists of all modules which have surjections from objects in $\add T$.
\end{example}

It is natural to consider the situation where $\XX_T$ is resolving and $\YY_T$ are coresolving, so that they both have $\Ext$-progenerators and $\Ext$-injective cogenerators. This leads to the notion of Wakamatsu tilting modules.
\begin{definition}
  A $\Lambda$-module $T \in \mod\Lambda$ is \emph{Wakamatsu tilting} it is self-orthogonal and $\Lambda \in \XX_T$ (or equivalently, if $\XX_T$ is a resolving subcategory of $\mod\Lambda$).
\end{definition}
It is known that this definition is self-dual:
\begin{proposition}[{\cite[Proposition 2.2]{BS}}]\label{prop:wak-dual}
  Let $T$ be a self-orthogonal module. Then $T$ is Wakamatsu tilting if and only if $D \Lambda \in \YY_T$,or equivalently, $\YY_T$ is a coresolving subcategory of $\mod\Lambda$.
\end{proposition}

Since the notion of Wakamatsu tilting modules is self-dual, we have a choice whether to use $\XX_T$ and ${}^\perp T$, or $\YY_T$ and $T^\perp$. In this paper, we regard a Wakamatsu tilting module as a generalization of tilting modules, so we mainly focus on $\YY_T$ and $T^\perp$, and omit the dual description of $\XX_T$ and ${}^\perp T$. However, since the category $\XX_\Lambda = \GP \Lambda$ of Gorenstein-projective modules is well-studied, we will use $\XX_T$ and ${}^\perp T$ when $T = \Lambda$.

For later use, we prepare the following observation of higher Ext groups, which we will use freely. The proof is straightforward since a resolving (resp. coresolving) subcategory is closed under taking syzygies (resp. cosyzygies).
\begin{lemma}
  Let $\EE$ be a resolving subcategory or a coresolving subcategory of $\mod\Lambda$. Then $\Ext^{>0}_\Lambda(\PP(\EE), \EE) = 0$ and $\Ext_\Lambda^{>0}(\EE, \II(\EE)) = 0$. Moreover, $\Ext_\EE^i(X, Y) \iso \Ext_\Lambda^i(X, Y)$ holds for $X, Y \in \EE$ when we regard $\EE$ as an exact category.
\end{lemma}

The following result from Wakamatsu tilting theory will be needed later.
\begin{lemma}[{\cite[Corollary 3.2, Theorem 4.2]{wakamatsu2}}]\label{lem:wak-duality}
  Let $T$ be a Wakamatsu tilting $\Lambda$-module, and put $\Gamma := \End_\Lambda(T)$. Then $_\Gamma T$ is a Wakamatsu tilting left $\Gamma$-module, and the functors $\Hom_\Lambda(-, T)$ and $\Hom_\Gamma(-, T)$ induce a duality between exact categories $\XX_{T_\Lambda} \subseteq \mod\Lambda$ and $\XX_{_\Gamma T} \subseteq \mod\Gamma^{\op}$.
\end{lemma}

Finally, let us briefly recall the notion of finite covers of a subcategory and its relation to covariant finiteness.
\begin{definition}
  Let $\CC$ be a subcategory of $\mod\Lambda$.
  \begin{enumerate}
    \item $P \in \CC$ is a \emph{cover} of $\CC$ if for every $C \in \CC$ there exists a surjection $P' \defl C$ with $P' \in \add P$.
    \item $\CC$ has a \emph{finite cover} if there exists some cover $P \in \CC$ of $\CC$.
    \item A cover $P$ of $\CC$ is called a \emph{minimal cover} if, for every cover $Q$ of $\CC$, we have $\add P \subseteq \add Q$.
  \end{enumerate}
\end{definition}
We have the following basic results on these concepts.
\begin{proposition}[{\cite[Corollary 2.4, Proposition 3.7]{AS}}]\label{prop:cover-fact}
  Let $\CC$ be a subcategory of $\mod\Lambda$ closed under extensions and direct summands.
  \begin{enumerate}
    \item If $\CC$ has a finite cover, then $\CC$ has a minimal cover.
    \item If $\CC$ is covariantly finite in $\mod\Lambda$, then $\CC$ has a finite cover, so it has a minimal cover.
    \item Let $P$ be a minimal cover of $\CC$. Then $P$ is $\Ext$-projective in $\CC$.
  \end{enumerate}
\end{proposition}

\section{Projectively Wakamatsu tilting modules}
\subsection{Basic properties}

For a Wakamatsu tilting module $T$, we have two resolving subcategories $\XX_T \subseteq {}^\perp T$ and two coresolving subcategories $\YY_T \subseteq T^\perp$.   One of the main motivations of this paper is to investigate the conditions under which these subcategories are equal.
To make this precise, let us define the following class of modules, which is the main topic of this paper.
\begin{definition}
  A $\Lambda$-module $T$ is \emph{projectively Wakamatsu tilting} if $T$ is an $\Ext$-progenerator of $T^\perp$. Dually, $T$ is \emph{injectively Wakamatsu tilting} if $T$ is an $\Ext$-injective cogenerator of $^\perp T$.
\end{definition}
This simple definition rephrases Wakamatsu tilting modules which we are interested in:
\begin{proposition}\label{prop:wak-proj-char}
  The following conditions are equivalent for $T \in \mod\Lambda$:
  \begin{enumerate}
    \item $T$ is projectively Wakamatsu tilting.
    \item $T$ is self-orthogonal and satisfies $\YY_T = T^\perp$.
  \end{enumerate}
  In this case, $T$ is Wakamatsu tilting.
\end{proposition}
\begin{proof}
  (1) $\Rightarrow$ (2):
  Let $T$ be projectively Wakamatsu tilting. Since $T$ is an $\Ext$-progenerator of $T^\perp$, it follows that $T \in T^\perp$, so $T$ is self-orthogonal. Using the projective resolution in $T^\perp$, we obtain $T^\perp \subseteq \YY_T$, and thus $T^\perp = \YY_T$. Additionally, $D\Lambda \in T^\perp$ since $D\Lambda$ is injective, so we have $D\Lambda \in \YY_T$, which by Proposition \ref{prop:wak-dual} implies that $T$ is Wakamatsu tilting.

  (2) $\Rightarrow$ (1):
  This implication follows directly from Proposition \ref{prop:4cats-projinj} (4).
\end{proof}

\begin{example}\label{ex:tilt-is-wak-proj}
  Let $T$ be a tilting module. As previously mentioned in Example \ref{ex:tilt-wak-proj}, we have the equality $T^\perp = \YY_T$. Therefore, every tilting module is projectively Wakamatsu tilting. Dually, every cotilting module is injectively Wakamatsu tilting.
\end{example}

\begin{example}\label{ex:weakly-gor}
  Clearly, $\Lambda_\Lambda$ itself is Wakamatsu tilting. Therefore, $\XX_\Lambda = \GP \Lambda$ is a resolving subcategory of $\mod\Lambda$, which is Frobenius with $\PP(\GP\Lambda) = \II(\GP \Lambda) = \proj\Lambda$.
  According to \cite{RZ}, $\Lambda$ is called \emph{weakly Gorenstein} if $\GP \Lambda = {}^\perp \Lambda$. Therefore, $\Lambda$ is weakly Gorenstein if and only if $\Lambda_\Lambda$ is injectively Wakamatsu tilting.
  Many classes of algebras are weakly Gorenstein (see \cite{ma2, RZ}), but there are also examples of algebras which are not weakly Gorenstein \cite{JS, ma1}.
  Therefore, such an algebra yields an example of a Wakamatsu tilting module which is not injectively Wakamatsu tilting.
\end{example}

A similar yet subtly different generalization of tilting modules was considered by Auslander and Reiten \cite[p. 144]{applications}.
\begin{definition}
  A $\Lambda$-module $T$ is an \emph{Auslander-Reiten tilting module}, abbreviated as \emph{AR tilting module}, if $\PP(T^\perp) = \add T$ and $T^\perp$ is covariantly finite.
\end{definition}
We omit the details on the dual notion, AR cotilting modules. The relationship between tilting modules, AR tilting modules, and projectively Wakamatsu tilting modules is summarized as follows.
\begin{proposition}\label{prop:tilt-rels}
  For $T \in \mod\Lambda$, the following hold:
  \begin{enumerate}
    \item If $T$ is tilting, then $T$ is AR tilting.
    \item If $T$ is AR tilting, then $T$ is projectively Wakamatsu tilting.
    \item Suppose that $\pd T_\Lambda$ is finite. Then $T$ is tilting if and only if AR tilting if and only if projectively Wakamatsu tilting.
  \end{enumerate}
\end{proposition}
\begin{proof}
  (1)
  This is proven in \cite[Theorem 5.5]{applications}.

  (2) In general, it is known that a covariantly finite extension-closed subcategory of $\mod\Lambda$ has enough $\Ext$-projectives (see e.g. \cite[Proposition 1.1]{Moh}). Therefore, if $T$ is AR tilting, then $T^\perp$ has enough $\Ext$-projectives with $\PP(T^\perp) = \add T$, which shows that $T$ is projectively Wakamatsu tilting.

  (3)
  It is enough to show that if $\pd T_\Lambda$ is finite and $T$ is projectively Wakamatsu tilting, then it is tilting.
  Let $d := \pd T_\Lambda$. We will first show that $T^\perp$ is covariantly finite.
  For any $X \in \mod\Lambda$, consider the injective resolution of $X$ and the $d$-th cosyzygy $\Sigma^d(X)$:
  \[
    \begin{tikzcd}
      0 \rar & X \rar & I^0 \rar & I^1 \rar & \cdots \rar & I^{d-1} \rar & \Sigma^d(X) \rar & 0
    \end{tikzcd}
  \]
  with $I^i \in \inj\Lambda$ for all $i$. Then we have $\Ext_\Lambda^{>d}(T, X) = \Ext_\Lambda^{>0}(T, \Sigma^d(X)) = 0$ by $\pd T_\Lambda = d$. Therefore, we obtain $\Sigma^d(X) \in T^\perp$, so the above exact sequence gives a finite coresolution of $X$ by $T^\perp$.
  Since $T^\perp$ is coresolving and has an $\Ext$-progenerator $T$, the Auslander-Buchweitz theory implies that $T^\perp$ is covariantly finite \cite[Theorem 2.3]{AB}.
  Then, the result by Auslander-Reiten \cite[Theorem 5.5]{applications} shows that $T^\perp$ should coincide with $T'^\perp$ for some tilting module $T'$. Then we have $\add T = \PP(T^\perp) = \PP(T'^\perp) = \add T'$, so $T$ is a tilting module.
\end{proof}
We remark that whether a Wakamatsu tilting module of finite projective dimension is tilting is a famous open problem called the \emph{Wakamatsu Tilting conjecture}.

To summarize, we have the following hierarchy.
\[
  \begin{tikzcd}[sep = tiny]
    \{ \text{tilting} \} \rar[symbol=\subseteq] &
    \{ \text{AR tilting} \} \rar[symbol=\subseteq] &
    \{ \text{projectively Wakamatsu tilting} \} \ar[rd, symbol=\subseteq]\\
    & & & \{ \text{Wakamatsu tilting} \} \\
    \{ \text{cotilting} \} \rar[symbol=\subseteq] &
    \{ \text{AR cotilting} \} \rar[symbol=\subseteq] &
    \{ \text{injectively Wakamatsu tilting} \} \ar[ru, symbol=\subseteq]
  \end{tikzcd}
\]

The inclusion $\{ \text{projectively Wakamatsu tilting} \} \subseteq \{ \text{Wakamatsu tilting} \}$ can be strict, as seen in Example \ref{ex:weakly-gor}.
The author is not aware of any other examples of Wakamatsu tilting modules which are not projectively Wakamatsu tilting (or injectively Wakamatsu tilting) except for this example coming from Gorenstein homological algebra, thus leading to the following question:
\begin{question}
  Are there systematic ways to construct Wakamatsu tilting modules which are not projectively Wakamatsu tilting modules?
\end{question}
Also, the author does not know whether AR tilting modules and projectively Wakamatsu tilting modules coincide. Specifically:
\begin{question}
  Are there any projectively Wakamatsu tilting modules which are not AR tilting modules? In other words, if $T$ is an $\Ext$-progenerator of $T^\perp$, does it follow that $T^\perp$ is covariantly finite?
\end{question}

As we will see later in Theorem \ref{thm:main2} and Remark \ref{rem:ar}, if $\Lambda$ is representation-finite, then the classes of AR tilting, projectively Wakamatsu tilting, and Wakamatsu tilting modules all coincide, so all modules in the above figure except for tilting and cotilting modules are the same.
Furthermore, Corollary \ref{cor:IG} states that if $\Lambda$ is a representation-finite Iwanaga-Gorenstein algebra, then every Wakamatsu tilting module is tilting, so all modules in the above figure coincide.

\subsection{Bongartz completion of a self-orthogonal module}
In this subsection, we first consider whether a given self-orthogonal module can be completed into a projectively Wakamatsu tilting module, and then we present our main result.
The notion of Bongartz completion is the basic tool we use to address this question.
\begin{definition}
  Let $U$ be a self-orthogonal module. A \emph{Bongartz completion of $U$} is a projectively Wakamatsu tilting module $T$ satisfying $U^\perp = T^\perp (= \YY_T)$.
\end{definition}
Since $U \in \PP(U^\perp)$ for a self-orthogonal module $U$, it easily follows that $U \in \add T$ if $T$ is a Bongartz completion of $U$.
The definition of Bongartz completion requires $T$ to be projectively Wakamatsu tilting, not merely Wakamatsu tilting.
In addition, if a Bongartz completion of $U$ exists, it is unique up to direct summands since $\add T = \PP(T^\perp) = \PP(U^\perp)$ holds.

This notion of Bongartz completion for projectively Wakamatsu tilting modules is compatible with the usual Bongartz completion for tilting modules, as the following proposition shows.
\begin{proposition}\label{prop:bon-compati}
  Let $U$ be a self-orthogonal module such that $\pd U_\Lambda$ is finite. Suppose that $T$ is a Bongartz completion of $U$ in the sense described above. Then $T$ is a tilting module.
\end{proposition}
\begin{proof}
  By Proposition \ref{prop:tilt-rels} (3), it suffices to show that $\pd T_\Lambda$ is finite. Let $d := \pd U_\Lambda$ and take any $X \in \mod\Lambda$. Then, for the $d$-th cosyzygy $\Sigma^d(X)$ of $X$, we have $\Ext_\Lambda^{>0}(U, \Sigma^d(X)) = \Ext_\Lambda^{>d}(U, X) = 0$, so $\Sigma^d(X) \in U^\perp$.
  Therefore, we obtain $\Sigma^d(X) \in T^\perp$ from $U^\perp = T^\perp$.
  Then, $\Ext_\Lambda^{d+1}(T, X) = \Ext_\Lambda^1(T, \Sigma^d(X)) = 0$ holds. Therefore, we conclude that $\pd X_\Lambda \leq d$.
\end{proof}
\begin{remark}
  It is well-known that if $\pd U_\Lambda \leq 1$ then $U$ has a Bongartz completion. However, for a general tilting module, there is an example of a self-orthogonal module $U$ of finite projective dimension such that a Bongartz completion of $U$ does not exist (see Example \ref{ex:mso-not-wak}).
\end{remark}

The following theorem answers the natural question of when a self-orthogonal module admits a Bongartz completion.
\begin{theorem}\label{thm:main}
  For a self-orthogonal module $U$, the following conditions are equivalent:
  \begin{enumerate}
    \item $U$ has a Bongartz completion.
    \item $U^\perp$ has a finite cover.
  \end{enumerate}
  Moreover, if $U^\perp$ has a minimal cover $M$, then $U \oplus M$ is a Bongartz completion of $U$.
\end{theorem}
\begin{proof}
  (1) $\Rightarrow$ (2):
  Let $T$ be a Bongartz completion of $U$, so $U^\perp = T^\perp$.
  Since $T$ is projectively Wakamatsu tilting, $T^\perp$ has an $\Ext$-progenerator $T$. In particular, $T$ is a cover of $T^\perp$, so $T^\perp = U^\perp$ has a finite cover.

  (2) $\Rightarrow$ (1):
  The proof of \cite[Theorem 1]{RS} serves as the main inspiration for the following argument. Let $M$ be a cover of $U^\perp$. By Proposition \ref{prop:cover-fact}, we may assume that $M$ is a minimal cover, and in particular, $M \in \PP(U^\perp)$. We will show that $U \oplus M$ is a Bongartz completion of $U$.

  First, we show that $U \oplus M$ is an $\Ext$-progenerator of $U^\perp$. Let $X \in U^\perp$. Take a right $U$-approximation $f \colon U_X \to X$. In addition, there exists a surjection $g \colon M' \defl X$ with $M' \in \add M$ since $M$ is a cover of $U^\perp$.
  Then consider the following exact sequence in $\mod\Lambda$:
  \begin{equation}\label{eq:proj-resol}
    \begin{tikzcd}
      0 \rar & X' \rar & U_X \oplus M' \rar["{[f, g]}"] & X \rar & 0.
    \end{tikzcd}
  \end{equation}
  By applying $\Hom_\Lambda(U, -)$, it is straightforward to check that $X' \in U^\perp$ since $g$ is a right $U$-approximation and $U_X \oplus M', X \in U^\perp$.
  Therefore, all the terms in the above sequence belong to $U^\perp$. Moreover, since $U, M \in \PP(U^\perp)$, it follows that $U \oplus M$ is a $\Ext$-progenerator of $U^\perp$.

  Since $U^\perp$ is coresolving and $U \oplus M \in \PP(U^\perp)$, we have $\Ext_\Lambda^{>0}(U \oplus M, X) = 0$ for every $X \in U^\perp$, that is, $U^\perp \subseteq (U \oplus M)^\perp$. Hence, $U^\perp = (U \oplus M)^\perp$ because the reverse containment is trivial. Therefore, $U \oplus M$ is an $\Ext$-progenerator of $(U \oplus M)^\perp = U^\perp$, so $U \oplus M$ is projectively Wakamatsu tilting and a Bongartz completion of $U$.
\end{proof}
We can deduce the following result of tilting theory immediately.
\begin{corollary}
  Let $T$ be a self-orthogonal $\Lambda$-module of finite projective dimension. Then $T$ admits a Bongartz completion (as a tilting module) if and only if $T^\perp$ has a finite cover.
\end{corollary}
\begin{proof}
  This follows directly from Theorem \ref{thm:main} and Proposition \ref{prop:bon-compati}, since for such a module $T$, a Bongartz completion must be tilting.
\end{proof}
It is worth noting that a similar statement is shown in \cite[Proposition 5.12]{applications}: such $T$ admits a Bongartz completion if and only if $T^\perp$ is covariantly finite. The above result is slightly stronger than this, since every covariantly finite subcategory has a finite cover.

We have the following consequence for the representation-finite case.
\begin{corollary}\label{cor:bon-for-rep-fin}
  Let $U$ be a self-orthogonal module such that $\#\ind (U^\perp)$ is finite (e.g. $\Lambda$ is representation-finite). Then $U$ has a Bongartz completion. In particular, there exists some $M \in \mod\Lambda$ such that $U \oplus M$ is a projectively Wakamatsu tilting module.
\end{corollary}
\begin{proof}
  This is an immediate consequence of Theorem \ref{thm:main}, since if $\#\ind (U^\perp)$ is finite, then $U^\perp$ has a finite cover.
\end{proof}

Also, Theorem \ref{thm:main} gives a criterion for determining when a self-orthogonal module is projectively Wakamatsu tilting.
\begin{corollary}\label{cor:wak-proj-char2}
  For a self-orthogonal module $T$, the following conditions are equivalent.
  \begin{enumerate}
    \item $T$ is projectively Wakamatsu tilting.
    \item $T^\perp = \YY_T$ holds.
    \item $T$ is a cover of $T^\perp$.
  \end{enumerate}
\end{corollary}
\begin{proof}
  (1) $\Leftrightarrow$ (2): This follows from Proposition \ref{prop:wak-proj-char}.

  (2) $\Rightarrow$ (3):
  This is obvious since $T$ is a cover of $\YY_T$.

  (3) $\Rightarrow$ (1):
  Since $T$ is a cover of $T^\perp$, the minimal cover $M$ of $T^\perp$ belongs to $\add T$. Consequently, the Bongartz completion of $T$ which we construct in the proof of Theorem \ref{thm:main} is $T \oplus M$, which is equivalent to $T$ up to direct summands. Therefore, $T$ is projectively Wakamatsu tilting.
\end{proof}
\begin{remark}
  We can prove this directly without using Theorem \ref{thm:main}: indeed, for any module $X \in T^\perp$, take a minimal right $T$-approximation $f$. Then, by Wakamatsu's lemma, $\ker f$ belongs to $T^\perp$. Moreover, $f$ is surjective since $T$ covers $T^\perp$. Thus, by repeating this process, we obtain $T^\perp \subseteq \YY_T$.
\end{remark}

By applying the dual of Corollary \ref{cor:wak-proj-char2} to $\Lambda_\Lambda$, we can deduce the following result on Gorenstein-projective modules, which was shown in \cite[Theorem 1.2]{RZ} by a different method.
\begin{corollary}
  The following conditions are equivalent:
  \begin{enumerate}
    \item $\Lambda$ is weakly Gorenstein, that is, $\XX_\Lambda = {}^\perp \Lambda$ holds.
    \item Every module $X$ in ${}^\perp \Lambda$ is torsionless, that is, there exists an injection $X \hookrightarrow P$ with some $P \in \proj\Lambda$.
  \end{enumerate}
\end{corollary}
\begin{proof}
  The dual of Corollary \ref{cor:wak-proj-char2} shows that $\Lambda$ is injectively Wakamatsu tilting if and only if $\Lambda$ is a cocover of $^\perp \Lambda$. Since the latter condition is precisely (2), we obtain the equivalence.
\end{proof}

Also, by applying Corollary \ref{cor:wak-proj-char2} to modules of finite projective dimension, we obtain the following characterization of tilting modules.
\begin{corollary}
  Let $T$ be a self-orthogonal module of finite projective dimension. Then the following conditions are equivalent:
  \begin{enumerate}
    \item $T$ is a tilting module.
    \item $T^\perp = \YY_T$ holds.
    \item $T$ is a cover of $T^\perp$.
  \end{enumerate}
\end{corollary}
\begin{proof}
  This follows from Corollary \ref{cor:wak-proj-char2} and Proposition \ref{prop:tilt-rels} (3).
\end{proof}

\subsection{Maximal self-orthogonal modules and Wakamatsu tilting modules}
In the classical tilting theory, the number of indecomposable direct summands of any tilting module is equal to $|\Lambda|$, so it is not possible to complete a tilting module into another tilting module in a non-trivial way.
It is reasonable to expect similar behavior for Wakamatsu tilting modules.
For instance, if we could apply Theorem \ref{thm:main} to a Wakamatsu tilting module $T$ which is not projectively Wakamatsu tilting, then we would obtain another projectively Wakamatsu tilting module $T'$ with $T \in \add T'$. We conjecture that such a thing is not possible.

This leads to the following notion, which plays a central role in the remainder of this paper.
\begin{definition}
  A $\Lambda$-module $T$ is \emph{maximal self-orthogonal} if it satisfies the following conditions.
  \begin{enumerate}
    \item $T$ is self-orthogonal.
    \item If $T \oplus M$ is self-orthogonal for $M \in \mod\Lambda$, then $M \in \add T$ holds.
  \end{enumerate}
\end{definition}
Therefore, the initial question of this subsection can be reformulated as follows: is every Wakamatsu tilting module maximal self-orthogonal?
Unfortunately, we cannot answer this question in general (Conjecture \ref{conj:msoc}). However, our main result Theorem \ref{thm:main2} provides a partial answer under a certain finiteness assumption.

First, we make the following observation about the number of indecomposable $\Ext$-projective and $\Ext$-injective objects of a subcategory of $\mod\Lambda$.
\begin{lemma}\label{lem:exact-cat}
  Let $\EE$ be a subcategory of $\mod\Lambda$ which is closed under extensions and direct summands. Suppose that $\#\ind\EE$ is finite and that $P$ and $I$ satisfy $\add P = \PP(\EE)$ and $\add I = \II(\EE)$. Then $|P| = |I|$.
\end{lemma}
\begin{proof}
  Observe that $\EE$ is functorially finite in $\mod\Lambda$, as $\#\ind\EE$ is finite.
  Therefore, \cite{almost} implies that $\EE$ has almost split sequences.
  This implies that there is a bijection between $\ind\EE \setminus \ind \PP(\EE)$ and $\ind\EE \setminus \ind \II(\EE)$.
  Since $\ind\EE$ is finite, we can count the number of objects, so we obtain $\# \ind \EE - \# \ind (\PP(\EE)) = \# \ind \EE - \# \ind (\II(\EE))$, which shows that $|P| = \# \ind (\PP(\EE)) = \# \ind (\II(\EE)) = |I|$.
\end{proof}

Before proving the main theorem, we prepare the following lemma.
\begin{lemma}\label{lem:wak-number}
  Let $U$ be a self-orthogonal module such that $\#\ind (U^\perp)$ is finite. Then there exists a Bongartz completion $T$ of $U$, which satisfies $|T| = |\Lambda|$. In particular, we have $|U| \leq |\Lambda|$.
\end{lemma}
\begin{proof}
  By Corollary \ref{cor:bon-for-rep-fin}, there is a Bongartz completion $T$ of $U$. Then $U^\perp = T^\perp$ has an $\Ext$-progenerator $T$ and an $\Ext$-injective cogenerator $D\Lambda$.
  Therefore, Lemma \ref{lem:exact-cat} shows that $|T| = |D \Lambda| = |\Lambda|$.
  The last statement follows from the fact that $U \in \add T$.
\end{proof}
Note that this proves the Boundedness conjecture (Conjecture \ref{conj:bounded}) under the finiteness assumption.
Now, we present the main result of this paper.
\begin{theorem}\label{thm:main2}
  Let $T$ be a $\Lambda$-module such that $\#\ind (T^\perp)$ is finite (e.g. $\Lambda$ is representation-finite). Then the following are equivalent.
  \begin{enumerate}
    \item $T$ is projectively Wakamatsu tilting.
    \item $T$ is Wakamatsu tilting.
    \item $T$ is self-orthogonal with $|T| = |\Lambda|$.
    \item $T$ is maximal self-orthogonal.
  \end{enumerate}
  In particular, we have $\YY_T = T^\perp$ in this case.
\end{theorem}
\begin{proof}
  (1) $\Rightarrow$ (2): This follows from Proposition \ref{prop:wak-proj-char}.

  (2) $\Rightarrow$ (3): Since $T$ is Wakamatsu tilting, $\YY_T$ has an $\Ext$-progenerator $T$ and an $\Ext$-injective cogenerator $D\Lambda$.
  By $\YY_T \subseteq T^\perp$, we can apply Lemma \ref{lem:exact-cat} to $\YY_T$, and we obtain $|T| = |D\Lambda| = |\Lambda|$.

  (3) $\Rightarrow$ (4):
  Suppose that $T \oplus M$ is self-orthogonal. Then, since $(T \oplus M)^\perp \subseteq T^\perp$, by applying Lemma \ref{lem:wak-number} to $T \oplus M$, we obtain $|T \oplus M| \leq |\Lambda|$. Thus, (3) implies $|\Lambda| = |T| \leq |T \oplus M| \leq |\Lambda|$, so $|T| = |T \oplus M|$. Therefore, we obtain $M \in \add T$.

  (4) $\Rightarrow$ (1):
  By Corollary \ref{cor:bon-for-rep-fin}, we obtain a Bongartz completion $T'$ of $T$. However, since $T$ is maximal self-orthogonal, we must have $\add T = \add T'$. Thus, $T$ is projectively Wakamatsu tilting since $T'$ is so.
\end{proof}

Theorem \ref{thm:main2} is useful for finding Wakamatsu tilting modules in $\mod \Lambda$ if $\Lambda$ is representation-finite. Indeed, for a given self-orthogonal module $T$, it is clearly easier to check whether $T$ is maximal self-orthogonal or $|T| = |\Lambda|$, compared to verifying that $\Lambda \in \XX_T$ or $D\Lambda \in \YY_T$. This is also helpful for computing  $\YY_T$ (and $\XX_T$) for a Wakamatsu tilting module $T$ since we get $\YY_T = T^\perp$ (and $\XX_T = {}^\perp T$) due to Theorem \ref{thm:main2}.

\begin{remark}\label{rem:ar}
  Let us briefly discuss AR tilting modules in Theorem \ref{thm:main2}. If $T$ is AR tilting, then it is projectively Wakamatsu tilting by Proposition \ref{prop:tilt-rels}. Conversely, if $T$ satisfies the equivalent conditions in Theorem \ref{thm:main2}, then it is AR tilting, since $\#\ind (T^\perp)$ is finite and thus $T^\perp$ is covariantly finite, and $\add T = \PP(T^\perp)$ holds as $T$ is projectively Wakamatsu tilting.
\end{remark}

By applying the dual of Theorem \ref{thm:main2} to $\Lambda_\Lambda$, we immediately obtain the following consequence in Gorenstein homological algebra.
\begin{corollary}\label{cor:weakly-gor}
  Suppose that $\# \ind ({}^\perp \Lambda)$ is finite.
  Then $\Lambda$ is weakly Gorenstein, that is, $\GP\Lambda = {}^\perp \Lambda$ holds.
\end{corollary}
\begin{proof}
  Since $\Lambda$ is Wakamatsu tilting, the dual of Theorem \ref{thm:main2} implies we have that $\Lambda_\Lambda$ is injectively Wakamatsu tilting.
  Then the dual of Proposition \ref{prop:wak-proj-char} shows $\XX_\Lambda = {}^\perp \Lambda$, that is, $\GP \Lambda = {}^\perp\Lambda$.
\end{proof}
This result is obtained in \cite[Corollary 5.11]{be} with a different proof. We note that the condition of finiteness of $\#\ind({}^\perp \Lambda)$ is weakened in \cite[Theorem 1.3]{RZ}.

\section{Binary relation on Wakamatsu tilting modules}
In this section, we introduce and discuss a binary relation on the set of Wakamatsu tilting modules, which generalizes the well-known poset structure on tilting modules.
\begin{definition}\label{def:binary}
  Let $\Lambda$ be an artin algebra.
  \begin{enumerate}
    \item $\wtilt\Lambda$ denotes the set of isomorphism classes of basic Wakamatsu tilting $\Lambda$-modules.
    \item $\tilt\Lambda$ denotes the set of isomorphism classes of basic tilting $\Lambda$-modules (of finite projective dimension).
    \item Let $T_1, T_2 \in \wtilt\Lambda$. We define $T_1 \geq T_2$ if $\Ext_\Lambda^{>0}(T_1, T_2) = 0$.
  \end{enumerate}
  We often identify modules with their representatives in $\wtilt\Lambda$.
\end{definition}
We note that a similar binary relation $\succeq$ is considered in \cite{wakamatsu2}: for $T_1,T_2\in\wtilt\Lambda$, we have $T_1 \succeq T_2$ if $T_1 \in \XX_{T_2}$ and $T_2 \in \YY_{T_1}$. This coincides with ours for the representation-finite case by Theorem \ref{thm:main2}.
\begin{remark}
  The poset $(\tilt\Lambda, \leq)$ has been studied extensively (e.g. \cite{HU-partial}), and $(\wtilt\Lambda, \leq)$ contains it as a full subposet.
  However, there are several important differences between $(\wtilt\Lambda, \leq)$ and $(\tilt\Lambda, \leq)$ as follows:
  \begin{itemize}
    \item While $(\tilt\Lambda, \leq)$ is always a poset, $(\wtilt\Lambda, \leq)$ is \emph{not necessarily a poset, even if $\Lambda$ is representation-finite}, since $\leq$ may not be transitive (see Example \ref{ex:not-poset}).
    \item If $T_1$ and $T_2$ are both tilting, then $T_1 \geq T_2$ is equivalent to $T_1^\perp \supseteq T_2^\perp$. However, $T_1 \geq T_2$ in $(\wtilt\Lambda, \leq)$ does not necessarily imply $T_1^\perp \supseteq T_2^\perp$ (see Example \ref{ex:poset-differ}).
  \end{itemize}
\end{remark}

The basic property of $(\wtilt\Lambda, \leq)$ is as follows.
\begin{proposition}
  Consider $\leq$ on $\wtilt\Lambda$.
  \begin{enumerate}
    \item $D\Lambda \leq T \leq \Lambda$ holds in $(\wtilt\Lambda, \leq)$ for every $T \in \wtilt\Lambda$.
    \item $\leq$ is reflexive.
    \item If $\Lambda$ is representation-finite, then $\leq$ is antisymmetric.
  \end{enumerate}
\end{proposition}
\begin{proof}
  (1) This follows directly from the definition.

  (2)  This is a consequence of $\Ext_\Lambda^{>0}(T, T) = 0$ for any $T \in \wtilt\Lambda$.

  (3) Suppose $T_1 \geq T_2 \geq T_1$. Then we obtain $\Ext_\Lambda^{>0}(T_1, T_2) = 0$ by $T_1 \geq T_2$, and also $\Ext_\Lambda^{>0}(T_2, T_1) = 0$ by $T_2 \geq T_1$. This implies that $T_1 \oplus T_2$ is self-orthogonal.
  However, since $\Lambda$ is representation-finite, $T_1$ and $T_2$ are maximal self-orthogonal by Theorem \ref{thm:main2}. We must thus have $T_1 \in \add T_2$ and $T_2 \in \add T_1$, which implies $T_1 \iso T_2$ since $T_1$ and $T_2$ are basic.
\end{proof}
As we can see from its proof, the antisymmetry of $\leq$ heavily depends on the fact that any Wakamatsu tilting $\Lambda$-module is maximal self-orthogonal.
Thus, the representation-infinite case remains unknown (Conjecture \ref{conj:msoc}).

The main result of this subsection is that $\tilt\Lambda \subseteq \wtilt\Lambda$ is upward-closed:
\begin{theorem}\label{thm:tilt-up-closed}
  Let $\Lambda$ be a representation-finite artin algebra. If $T \geq U$ holds in $\wtilt\Lambda$ and $U \in \tilt\Lambda$, then $T \in \tilt\Lambda$.
\end{theorem}
\begin{proof}
  Suppose $T \geq U$ in $\wtilt\Lambda$, that is, $\Ext_\Lambda^{>0}(T, U) = 0$.
  We will show that $U$ has a finite $T$-resolution. The proof is similar to \cite[Proposition 4.4]{MR}, but we provide it here for the reader's convenience.
  By Theorem \ref{thm:main2}, since $\Lambda$ is representation-finite, we have $U \in T^\perp = \YY_T$. Thus, we obtain the following exact sequence
  \[
    \begin{tikzcd}
      \cdots \rar["g_2"] & T_1 \rar["g_1"] & T_0 \rar["g_0"] & U \rar & 0
    \end{tikzcd}
  \]
  in $\mod\Lambda$ with $T_i \in \add T$ and $\im g_i \in T^\perp$ for all $i \geq 0$.
  Let $d := \pd U_\Lambda$, which is finite since $U$ is tilting. Consider the short exact sequence
  \begin{equation}\label{eq:d}
    \begin{tikzcd}
      0 \rar & \im g_{d+1} \rar & T_d \rar & \im g_d \rar & 0.
    \end{tikzcd}
  \end{equation}
  This corresponds to an element in $\Ext_\Lambda^1(\im g_d, \im g_{d+1})$.
  However, since $\im g_{d+1} \in T^\perp$, we can inductively show
  \[
    \Ext_\Lambda^1(\im g_d, \im g_{d+1})
    = \Ext_\Lambda^2(\im g_{d-1}, \im g_{d+1})
    = \cdots = \Ext_\Lambda^{d+1}(\im g_0, \im g_{d+1}),
  \]
  and the last term is $\Ext_\Lambda^{d+1}(U, \im g_{d+1})$, which is zero by $\pd U_\Lambda = d$.
  It follows that \eqref{eq:d} splits. Therefore, we obtain the following finite exact sequence by replacing $T_d$ with $\im g_d$:
  \begin{equation}\label{eq:coresol}
    \begin{tikzcd}
      0 \rar & T_d \rar["g_d"] & \cdots \rar["g_2"] & T_1 \rar["g_1"] & T_0 \rar["g_0"] & U \rar & 0,
    \end{tikzcd}
  \end{equation}
  with $T_i \in \add T$ and $\im g_i \in T^\perp$ for all $i \geq 0$. Moreover, since $T \in {}^\perp U$ and $^\perp U$ is resolving, $\im g_i \in T^\perp \cap {}^\perp U$ for all $i \geq 0$.

  Now, consider the subcategory $\EE := T^\perp \cap {}^\perp U$ of $\mod\Lambda$, which is closed under extensions and direct summands and contains both $T$ and $U$.
  Since $T$ is an $\Ext$-progenerator of $T^\perp$ and $^\perp U$ is resolving, it is straightforward to see that $T$ is an $\Ext$-progenerator of $\EE$, and similarly, $U$ is an $\Ext$-injective cogenerator of $\EE$.
  Then, \eqref{eq:coresol} implies that the ``projective dimension'' of $U$ in the exact category $\EE$ is finite.
  Then \eqref{eq:coresol} shows that the ``projective dimension'' of $U$ in an exact category $\EE$ is finite.
  Then, by Lemma \ref{lem:GSC} below, the ``injective dimension'' of $T$ in $\EE$ is finite, so we obtain the following exact sequence:
  \[
    \begin{tikzcd}
      0 \rar & T \rar & U^0 \rar & U^1 \rar & \cdots
      \rar & U^{d'} \rar & 0,
    \end{tikzcd}
  \]
  with $U^i \in \add U$ for all $i \geq 0$. Since $\pd U_\Lambda$ is finite, so is $\pd T_\Lambda$.
  Theorem \ref{thm:main2} implies that $T$ is projectively Wakamatsu tilting since $\Lambda$ is representation-finite. Then, Proposition \ref{prop:tilt-rels} (3) implies that $T$ is a tilting module.
\end{proof}

In the proof, we use the lemma below, which can be considered the Gorenstein Symmetry conjecture for exact categories with finitely many indecomposables. We omit the explanation for the terms related to exact categories, such as conflations and projective dimension.
\begin{lemma}\label{lem:GSC}
  Let $\EE$ be a Krull-Schmidt exact category with a progenerator $P$ and an injective cogenerator $I$. Suppose that $\#\ind\EE$ is finite, and that $\pd_\EE I$ is finite. Then $\id_\EE P$ is finite.
\end{lemma}
\begin{proof}
  We first make the following observation. Suppose that we have a conflation
  \[
    \begin{tikzcd}
      0 \rar & X \rar & Y \rar & Z \rar & 0
    \end{tikzcd}
  \]
  in $\EE$. If $X$ and $Y$ have finite projective dimension, then so does $Z$. This can be proved by a similar argument to the case of usual modules over rings.
  Moreover, since $\#\ind\EE$ is finite, we can define an integer $d$ by the following equation:
  \[
    d := \sup \{ \pd_\EE X \mid X \in \EE \text{ with } \pd_\EE X < \infty \}
  \]

  Now we will prove that $\id_\EE P$ is finite.
  Since $\EE$ has an injective cogenerator $I$, we can take the injective resolution of $P$, which consists of conflations of the form:
  \begin{equation}\label{eq:cosyzygy}
    \begin{tikzcd}
      0 \rar & \Sigma^i(P) \rar & I^i \rar & \Sigma^{i+1}(P) \rar & 0
    \end{tikzcd}
  \end{equation}
  with $\Sigma^0(P) = P$ and $I^i \in \add I$ for all $i$.
  Since $P$ and $I$ have finite projective dimension, by the earlier observation, we can inductively show that $\Sigma^i(P)$ has finite projective dimension for every $i$. Thus, we obtain $\pd \Sigma^i(P) \leq d$ for every $i$ by the definition of $d$.
  Now consider the following isomorphism:
  \[
    \Ext^1_\EE(\Sigma^{d+1}(P), \Sigma^d(P)) \iso \Ext_\EE^{d+1}(\Sigma^{d+1}(P), P).
  \]
  Since $\pd \Sigma^{d+1}(P) \leq d$, the right hand side vanishes. Therefore, the conflation in \eqref{eq:cosyzygy} for $i = d$ splits. This implies $\Sigma^d(P) \in \add I$, so we obtain an injective resolution of $P$ of length $d$ by replacing $I^d$ with $\Sigma^d(P)$. Thus, we conclude $\id_\EE P \leq d$.
\end{proof}

Applying Theorem \ref{thm:tilt-up-closed}, we obtain the following remarkable result about Iwanaga-Gorenstein algebras. Recall that an artin algebra $\Lambda$ is \emph{Iwanaga-Gorenstein} if $\id\Lambda_\Lambda$ and $\pd (D\Lambda)_\Lambda$ are both finite.
\begin{corollary}\label{cor:IG}
  Let $\Lambda$ be a representation-finite Iwanaga-Gorenstein artin algebra. Then every self-orthogonal $\Lambda$-module has finite projective dimension, and the following conditions are equivalent for $T \in \mod\Lambda$:
  \begin{enumerate}
    \item $T$ is tilting.
    \item $T$ is projectively Wakamatsu tilting.
    \item $T$ is Wakamatsu tilting.
  \end{enumerate}
\end{corollary}
\begin{proof}
  Since $\Lambda$ is representation-finite, Corollary \ref{cor:bon-for-rep-fin} implies that any self-orthogonal module is a direct summand of some projectively Wakamatsu tilting module.
  In addition, (1) $\Rightarrow$ (2) $\Rightarrow$ (3) holds in general by Proposition \ref{prop:wak-proj-char} and Example \ref{ex:tilt-is-wak-proj}. Therefore, it is sufficient to prove (3) $\Rightarrow$ (1).

  Let $T \in \wtilt\Lambda$, then we have $T \geq D\Lambda$. Moreover, since $\Lambda$ is Iwanaga-Gorenstein, it follows easily that $D\Lambda$ is tilting. Hence, Theorem \ref{thm:tilt-up-closed} implies that $T \in \tilt\Lambda$.
\end{proof}
Therefore, for a representation-finite Iwanaga-Gorenstein algebra, the notions of tilting, cotilting, projectively Wakamatsu tilting, injectively Wakamatsu tilting, and Wakamatsu tilting modules are equivalent.
Since the statement of the above corollary is simple, the following question arises:
\begin{question}
  Is there a direct proof of the fact that every self-orthogonal $\Lambda$-module over a representation-finite Iwanaga-Gorenstein algebra has finite projective dimension?
\end{question}

In addition, it is unknown whether similar results hold for representation-infinite Iwanaga-Gorenstein algebras, leading to the following conjecture.
\begin{conjecture}
  Let $\Lambda$ be an Iwanaga-Gorenstein artin algebra. Then every self-orthogonal $\Lambda$-module has finite projective dimension.
\end{conjecture}
\begin{remark}
  In relation to this conjecture, we should note two further homological conjectures.
  \begin{enumerate}
    \item The \emph{Tachikawa conjecture}: If $\Lambda$ is self-injective, then every self-orthogonal $\Lambda$-module is projective. This is equivalent to the above conjecture for the case where $\Lambda$ is self-injective, since a module is projective if and only if it has finite projective dimension.
    \item The \emph{Gorenstein-projective conjecture} \cite{gpc}: Every self-orthogonal Gorenstein-projective $\Lambda$-module is projective. If we consider an Iwanaga-Gorenstein algebra $\Lambda$, then this conjecture corresponds precisely to the case where only Gorenstein-projective modules are considered in the above question.
  \end{enumerate}
\end{remark}
\begin{conjecture}\label{conj:IGWTC}
  Let $\Lambda$ be an Iwanaga-Gorenstein artin algebra. Then any Wakamatsu tilting $\Lambda$-module is tilting.
\end{conjecture}
We remark that, if $\Lambda$ is Iwanaga-Gorenstein, then tilting and cotilting modules coincide. This can be seen as follows. Let $T$ be a cotilting module, so $\id T_\Lambda$ is finite. Since $\Lambda$ is Iwanaga-Gorenstein, this implies that $\pd T_\Lambda$ is finite and it then follows from \cite[Proposition 4.4]{MR} that $T$ is tilting.

In the rest of this section, we consider whether an analogous version of Theorem \ref{thm:tilt-up-closed} holds for representation-infinite algebras.
If $\Lambda$ is representation-infinite, then $T \geq U$ in $\wtilt\Lambda$ does not imply that $T \in \XX_U$ or $U \in \YY_T$. As such, it is reasonable to pose the following modified version:
\begin{conjecture}\label{conj:tilt-up-closed}
  Let $\Lambda$ be an artin algebra, and let $T, U \in \wtilt\Lambda$ such that $T \in \XX_U$ and $U \in \YY_T$. If $U$ is tilting, then so is $T$.
\end{conjecture}

We will show that this conjecture is implied by the \emph{Wakamatsu tilting conjecture}: every Wakamatsu tilting module of finite projective dimension is tilting.
To do this, we will use the following interpretation of the Wakamatsu tilting conjecture in terms of the Gorenstein symmetry conjecture in exact categories.
\begin{lemma}\label{lem:wtc-gsc}
  The Wakamatsu tilting conjecture is equivalent to the following conjecture: Let $\EE$ be a Hom-finite Krull-Schmidt exact $R$-category with a progenerator $P$ and an injective cogenerator $I$. If $\pd_\EE I$ is finite, then $\id_\EE P$ is also finite.
\end{lemma}
\begin{proof}
  Suppose that the latter conjecture holds and let $T$ be a Wakamatsu tilting $\Lambda$-module of finite projective dimension. Consider then the exact category $\XX_T$, which has a progenerator $\Lambda$ and an injective cogenerator $T$. Since $\XX_T$ is resolving, we conclude that $\pd_\EE T = \pd T_\Lambda$ is finite. Furthermore, $\id_\EE \Lambda$ is finite, due to the initial assumption.
  This implies that $\Lambda$ has a finite $T$-coresolution, so $T$ is a tilting $\Lambda$-module.

  Conversely, assume the Wakamatsu tilting conjecture, and let $\EE$ be an exact category satisfying the stated conditions.
  Let $\Gamma := \End_\EE(T)$, and consider the functor $F := \EE(P, -) \colon \EE \to \mod\Gamma$. Then, according to \cite[Theorem 3.3]{eno-wak}, $FI$ is a Wakamatsu tilting $\Gamma$-module, $F$ induces an exact equivalence $\EE \equi F(\EE)$, and $F(\EE)$ is a resolving subcategory of $\mod\Gamma$.

  Suppose that $\pd_\EE I$ is finite. Since $I$ has a finite $P$-resolution, applying $F$ to it yields that $FI$ has a finite $\Gamma$-resolution, that is, $FI$ has finite projective dimension. Therefore, the Wakamatsu tilting conjecture implies that $FI$ is a tilting $\Gamma$-module.
  Consequently, since $\Gamma_\Gamma = FP$, there is an exact sequence in $\mod\Gamma$ of the form
  \[
    \begin{tikzcd}
      0 \rar & FP \rar["Ff^0"] & FI^0 \rar["Ff^1"] & FI^1 \rar["Ff^2"] & \cdots \rar["Ff^d"] & FI^d \rar & 0
    \end{tikzcd}
  \]
  with $I^i \in \add I$ (since $F$ is fully faithful on $\EE$). Moreover, since $F(\EE)$ is a resolving subcategory of $\mod\Gamma$, it follows that $\im Ff^i \in F(\EE)$ for all $i \geq 0$.
  Therefore, the above exact sequence consists of short exact sequences in $F(\EE)$. In addition, since $F$ is an exact equivalence between $\EE$ and $F(\EE)$, there is also an exact sequence
  \[
    \begin{tikzcd}
      0 \rar & P \rar["f^0"] & I^0 \rar["f^1"] & I^1 \rar["f^2"] & \cdots \rar["f^d"] & I^d \rar & 0
    \end{tikzcd}
  \]
  in $\EE$ consisting of conflations in $\EE$. Thus, we can conclude that $\id_\EE P$ is finite.
\end{proof}
Now we can state the relation between the Wakamatsu tilting conjecture and our conjecture.
\begin{proposition}\label{prop:wtc-tilt-up}
  The Wakamatsu tilting conjecture implies Conjecture \ref{conj:tilt-up-closed}.
\end{proposition}
\begin{proof}
  Suppose that the Wakamatsu tilting conjecture is true.
  Consider the category $\EE := \YY_T \cap \XX_U$. The same argument as in Theorem \ref{thm:tilt-up-closed} shows that $\EE$ has an $\Ext$-progenerator $T$ and an $\Ext$-injective cogenerator $U$.
  In addition, since $\pd U_\Lambda$ is finite, the same argument as in Theorem \ref{thm:tilt-up-closed} shows that $U$ has a finite $T$-resolution in $\EE$.
  Therefore, $\pd_\EE T$ is finite.
  Lemma \ref{lem:wtc-gsc} then shows that $\id_\EE T$ is finite.
  Thus, there exists an exact sequence in $\mod\Lambda$
  \[
    \begin{tikzcd}
      0 \rar & T \rar & U^0 \rar & U^1 \rar & \cdots \rar & U^d \rar & 0
    \end{tikzcd}
  \]
  with $U^i \in \add U$ for all $i \geq 0$. As $\pd U_\Lambda$ is finite, we deduce that $\pd T_\Lambda$ is finite as well. Consequently, $T$ is a Wakamatsu tilting module of finite projective dimension. Therefore, the Wakamatsu tilting conjecture implies that $T$ is tilting.
\end{proof}

As a consequence of this, we can see that Conjecture \ref{conj:IGWTC} follows from the Wakamatsu tilting conjecture.
\begin{corollary}
  The Wakamatsu tilting conjecture implies Conjecture \ref{conj:IGWTC}.
\end{corollary}
\begin{proof}
  Let $\Lambda$ be an Iwanaga-Gorenstein algebra and $T$ a Wakamatsu tilting $\Lambda$-module. It is clear that $T \in \XX_{D\Lambda}$, and since $T$ is Wakamatsu tilting, $D\Lambda \in \YY_T$ (Proposition \ref{prop:wak-dual}).
  Moreover, $D\Lambda$ is tilting since $\Lambda$ is Iwanaga-Gorenstein.
  Therefore, Proposition \ref{prop:wtc-tilt-up} shows that the Wakamatsu tilting conjecture implies that $T$ is tilting.
\end{proof}

\section{Conjectures on self-orthogonal modules}\label{sec:conj}
In Theorem \ref{thm:main2}, we provided several characterizations of Wakamatsu tilting modules under the assumption that $\ind (T^\perp)$ is finite.
For a general $\Lambda$, some of these conditions are not equivalent.
Indeed, as mentioned in Example \ref{ex:weakly-gor}, there is a Wakamatsu tilting module which is not projectively Wakamatsu tilting, and also Example \ref{ex:mso-not-wak} shows that there is a maximal self-orthogonal module which is not Wakamatsu tilting.

Nevertheless, several open conjectures exist regarding the relationship between self-orthogonal modules and Wakamatsu tilting modules, as follows.
\begin{conjecture}[Boundedness conjecture]\label{conj:bounded}
  If $U$ is self-orthogonal, then $|U| \leq |\Lambda|$ holds.
\end{conjecture}
\begin{conjecture}[Proj=Inj conjecture]
  If $T$ is Wakamatsu tilting, then $|T| = |\Lambda|$ holds.
\end{conjecture}
\begin{conjecture}[Maximal self-orthogonal conjecture]\label{conj:msoc}
  Every Wakamatsu tilting module is maximal self-orthogonal.
\end{conjecture}
\begin{conjecture}[Weak maximal self-orthogonal conjecture]
  Every projectively Wakamatsu tilting module is maximal self-orthogonal.
\end{conjecture}

In addition, we recall the following two famous homological conjectures.
\begin{itemize}
  \item The \emph{Auslander-Reiten conjecture}: $\Lambda_\Lambda$ is maximal self-orthogonal.
  \item The \emph{Generalized Nakayama conjecture}: Every indecomposable injective module appears in the minimal injective resolution of $\Lambda$.
\end{itemize}

We abbreviate these conjectures as (BC), (PIC), (MSOC), (wMSOC), (ARC), and (GNC) respectively.
In view of Theorem \ref{thm:main2}, (PIC) states that (2) $\Rightarrow$ (3) holds, (MSOC) states that (2) $\Rightarrow$ (4) holds, and (wMSOC) states that (1) $\Rightarrow$ (4) holds.

The name \emph{Proj=Inj conjecture (PIC)} is derived from the following observation, which relates it to general subcategories of $\mod\Lambda$.
\begin{proposition}\label{prop:PIC-equiv}
  (PIC) for every artin algebra is equivalent to the following conjecture:
  Let $\CC$ be a subcategory of $\mod\Lambda$ which is closed under extensions and direct summands. Suppose that $\CC$ has an $\Ext$-progenerator $P$ and an $\Ext$-injective cogenerator $I$. Then $|P| = |I|$ holds.
\end{proposition}
\begin{proof}
  Let $T$ be a Wakamatsu tilting module. Then $\YY_T$ has an $\Ext$-progenerator $T$ and an $\Ext$-injective cogenerator $D\Lambda$. Thus, the stated conjecture implies that $|T| = |D\Lambda| =  |\Lambda|$.

  Conversely, let $\CC$ be a subcategory of $\mod\Lambda$ satisfying the stated condition, and put $\Gamma := \End_\Lambda(P)$ and $F := \Hom_\Lambda(P,-) \colon \mod\Lambda \to \mod \Gamma$.
  Thus, \cite[Theorem 3.3]{eno-wak} shows that $FI$ is a Wakamatsu tilting $\Gamma$-module and $F$ induces equivalences $\add P \equi \proj \Gamma$ and $\add I \equi \add FI$.
  Then, (PIC) will imply that $|\Gamma| = |FI|$, which shows that $|P| = |\Gamma| = |FI| = |I|$.
\end{proof}

\begin{remark}
  Here we provide a list of papers studying these conjectures, though by no means exhaustive.
  \begin{itemize}
    \item (BC) appears in various publications, such as \cite{happel, HU}. The name originates from \cite{happel}.
    \item (PIC) can be found in \cite{BS}, whereas the variant of Proposition \ref{prop:PIC-equiv} appears in \cite{AS}.
    \item (ARC) and (GNC) were formulated by Auslander and Reiten \cite{arc}. These two conjectures have become two of the most widely recognised homological conjectures.
  \end{itemize}
\end{remark}

As we have seen in Lemma \ref{lem:wak-number}, the Boundedness conjecture holds under the finiteness assumption. For the convenience of the reader, we give another simple proof of this without using Bongartz completion.
\begin{proposition}\label{prop:bounded}
  Let $U$ be a self-orthogonal module such that $\#\ind(U^\perp)$ is finite. Then $|U| \leq |\Lambda|$ holds.
\end{proposition}
\begin{proof}
  Since $\#\ind(U^\perp)$ is finite, we can apply Lemma \ref{lem:exact-cat} to $U^\perp$ to obtain $\# \ind \PP(U^\perp) = \# \ind \II(U^\perp)$. On the other hand, we have $\II(U^\perp) = \add D\Lambda$ and $U \in \PP(U^\perp)$ by Proposition \ref{prop:4cats-projinj}.
  Therefore, we obtain $|U| \leq \#\ind\PP(U^\perp) = |D\Lambda| = |\Lambda|$.
\end{proof}

The relationship between these conjectures is summarized as the following figure.
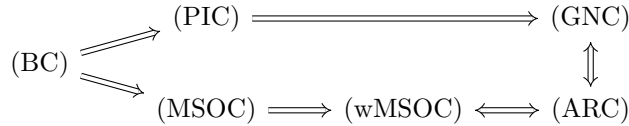
\[
  \begin{tikzcd}[row sep=0]
    & \text{(PIC)} \ar[rr, Rightarrow]& & \text{(GNC)} \ar[dd, Leftrightarrow] \\
    \text{(BC)} \ar[ru, Rightarrow] \ar[rd, Rightarrow] \\
    & \text{(MSOC)} \rar[Rightarrow] & \text{(wMSOC)} \rar[Leftrightarrow] & \text{(ARC)}
  \end{tikzcd}
\]
In particular, we will show that (wMSOC) is equivalent to (ARC).
The equivalence of (ARC) and (GNC) is shown in \cite{arc}, but some implications hold at the level of individual algebras, we consider both (ARC) and (GNC). Note that the implication (MSOC) $\Rightarrow$ (wMSOC) is clear from Proposition \ref{prop:wak-proj-char}.
\begin{proposition}\label{prop:conj-rel}
  We have the following implications:
  \begin{enumerate}
    \item (BC) implies (PIC).
    \item (BC) implies (MSOC).
    \item If $\Lambda$ satisfies either (BC), (PIC), or (wMSOC), then it satisfies (ARC).
    \item If $\Lambda$ satisfies either (PIC) or (MSOC), then it satisfies (GNC).
    \item (wMSOC) is equivalent to (ARC).
  \end{enumerate}
\end{proposition}
\begin{proof}
  (1)
  Assume (BC), and let $T$ be a Wakamatsu tilting $\Lambda$-module with $|T| \neq |\Lambda|$. By (BC), we obtain $|T| < |\Lambda|$.
  Let $\Gamma := \End_\Lambda(T)$.
  Then Lemma \ref{lem:wak-duality} shows that $_\Gamma T$ is a Wakamatsu tilting left $\Gamma$-module, and that $\Hom_\Lambda(-, {}_\Gamma T_\Lambda)$ induces dualities $\proj\Lambda \equi \add (_\Gamma T)$ and $\add(T_\Lambda) \equi \proj \Gamma^{\op}$.
  In particular, we obtain $|\Gamma| = |T_\Lambda| < |\Lambda| = |{}_\Gamma T|$, which contradicts (BC) since $_\Gamma T$ is self-orthogonal.

  (2)
  Assume (BC), and let $T$ be a Wakamatsu tilting module with $T \oplus M$ being self-orthogonal. Then (PIC) holds by (1), so we have $|T| = |\Lambda|$. Since $T \oplus M$ is self-orthogonal, (BC) implies $|T \oplus M| \leq |\Lambda|$.
  Thus, we have $|\Lambda| = |T| \leq |T \oplus M| \leq |\Lambda|$, which implies $|T \oplus M| = |T|$. This shows that $M \in \add T$, and thus $T$ is maximal self-orthogonal.

  (3)
  Assume (BC) or (PIC) for $\Lambda$, and suppose that $\Lambda \oplus M$ is self-orthogonal. Then (BC) will imply $|\Lambda \oplus M| \leq |\Lambda|$, which in turn implies $M \in \add\Lambda$.
  Observe that $\Lambda \oplus M$ is Wakamatsu tilting since we have $\Lambda \in \XX_{\Lambda \oplus M}$ by the trivial sequence $0 \to \Lambda = \Lambda \to 0$. Then (PIC) will imply $|\Lambda \oplus M| = |\Lambda|$, so $M \in \add \Lambda$.

  Finally, since $\Lambda_\Lambda$ is projectively Wakamatsu tilting, (wMSOC) for $\Lambda$ clearly implies (ARC) for $\Lambda$.

  (4)
  Let $Q$ be the direct sum of all non-isomorphic indecomposable injective modules appearing in the minimal injective resolution of $\Lambda$, so $Q \in \inj\Lambda = \add D\Lambda$. It suffices to show that $\add Q = \add D\Lambda$.
  Since $Q$ is injective, it is self-orthogonal, and the construction of $Q$ implies that $\Lambda \in \XX_Q$. Therefore, $Q$ is a Wakamatsu tilting module.
  If $\Lambda$ satisfies (PIC), then $|Q| = |\Lambda| = |D\Lambda|$, which implies that $\add Q = \add D\Lambda$.
  On the other hand, if $\Lambda$ satisfies (MSOC), then $Q$ is a maximal self-orthogonal module. Now since $Q \oplus D\Lambda$ is self-orthogonal, we obtain $\add Q = \add D\Lambda$.

  (5) By (4), it suffices to show that (ARC) implies (wMSOC).
  Let $T$ be a projectively Wakamatsu tilting $\Lambda$-module with $T \oplus M$ being self-orthogonal.
  Consider the category $T^\perp$, which has an $\Ext$-progenerator $T$ since $T$ is projectively Wakamatsu tilting.
  We also have $T \oplus M \in T^\perp$.

  Define $\Gamma := \End_\Lambda (T)$ and a functor $F := \Hom_\Lambda(T, -) \colon \mod \Lambda \to \mod\Gamma$. Then $F$ induces an exact equivalence $T^\perp \simeq F(T^\perp)$ and $F(T^\perp)$ is a resolving subcategory of $\mod\Gamma$ by \cite[Proposition 2.8]{eno-wak}. Since $T^\perp$ and $F(T^\perp)$ are coresolving and resolving subcategories of $\mod\Lambda$ and $\mod\Gamma$ respectively, higher Ext-groups inside these exact categories are the same as those in $\mod\Lambda$ and $\mod\Gamma$. In particular, $F$ preserves all higher Ext-groups. Therefore, $\Ext^{>0}_\Gamma(F(T \oplus M), F(T \oplus M)) = 0$. Since $F(T \oplus M) = \Gamma \oplus FM$, the Auslander-Reiten conjecture for $\Gamma$ implies $FM \in \proj\Gamma = \add (FT)$. Then, since $F$ is fully faithful on $T^\perp$, we obtain $M \in \add T$. Thus, $T$ is maximal self-orthogonal.
\end{proof}

As mentioned in the introduction, this proposition provides alternative proof of (ARC) and (GNC) for representation-finite algebras. Indeed, if $\Lambda$ is a representation-finite artin algebra, then Theorem \ref{thm:main2} shows that $\Lambda$ satisfies (MSOC), so it satisfies (ARC) and (GNC) by Proposition \ref{prop:conj-rel} (3) and (4).

Finally, we propose an additional conjecture regarding Theorem \ref{thm:main2} (3).
\begin{conjecture}
  Let $T$ be a self-orthogonal $\Lambda$-module. If $|T| = |\Lambda|$, then $T$ is Wakamatsu tilting.
\end{conjecture}
A similar conjecture concerning tilting modules has previously been raised in various sources, such as \cite{RS}: a self-orthogonal $\Lambda$-module of finite projective dimension over $\Lambda$ satisfying $|T| = |\Lambda|$ is tilting.

\section{Examples}\label{sec:example}
Throughout this section, we let $k$ be a field.
First, we provide some examples of projectively Wakamatsu tilting = Wakamatsu tilting modules over representation-finite algebras, which coincide by Theorem \ref{thm:main2}.
More precisely, we give examples of $(\wtilt\Lambda, \leq)$ (see Definition \ref{def:binary}). If $(\wtilt\Lambda, \leq)$ is a poset, we will illustrate it using its Hasse quiver, i.e., by drawing an arrow from $T_1$ to $T_2$ if $T_1 > T_2$ and there is no $T' \in \wtilt\Lambda$ satisfying $T_1 > T' > T_2$. Tilting modules in the Hasse quiver are represented by rectangles, while cotilting modules are represented by circles.
To calculate these examples, we used the computer program \cite{FD-applet}, which was developed by the author.

We note that, as Corollary \ref{cor:IG} shows, all Wakamatsu tilting modules are both tilting and cotilting if $\Lambda$ is representation-finite and Iwanaga-Gorenstein.
Therefore, to provide examples of Wakamatsu tilting modules which are neither tilting nor cotilting (referred to as \emph{non-(co)tilting}), we must consider non-Iwanaga-Gorenstein algebras.

To represent Nakayama algebras, we will use \emph{Kupisch series}, which is the series of dimensions of the indecomposable projective modules $P(1), P(2), P(3), \dots$ for Nakayama algebras with the quiver $1 \rightarrow 2 \rightarrow 3 \rightarrow \cdots$.
\begin{example}\label{ex:poset-differ}
  A Nakayama algebra $\Lambda$ with the smallest dimension admitting non-(co)tilting Wakamatsu tilting modules has Kupisch series [3, 4, 4, 4], which is given by the quiver
  \[
    \begin{tikzcd}
      1 \rar["a"] & 2 \dar["b"] \\
      4 \uar["d"] & 3 \lar["c"]
    \end{tikzcd}
  \]
  with relations $abc = bcda = cdab = 0$. There are 3 indecomposable projective-injective $\Lambda$-modules, namely $P(2)$, $P(3)$, and $P(4)$. As every Wakamatsu tilting module $T$ is maximal self-orthogonal, we have $P(2), P(3), P(4) \in \add T$. Moreover, $|T| = |\Lambda| = 4$ must hold, so a Wakamatsu tilting module $T$ is uniquely determined by an indecomposable non-projective-injective self-orthogonal module $X$.
  There are 6 such modules $X$, and $(\wtilt\Lambda, \leq)$ defines a poset. The figure below represents the Hasse quiver of $(\wtilt\Lambda, \leq)$, where we only show $X$.
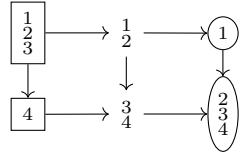
  \[
    \begin{tikzcd}[row sep=tiny]
      |[draw, rectangle]| \sst{1 \\ 2 \\ 3} \dar \rar
      & \sst{1 \\ 2} \rar \dar
      & |[draw, ellipse, inner sep=2]|\sst{1} \dar \\
      |[draw, rectangle]| \sst{4} \rar
      & \sst{3 \\ 4} \rar
      & |[draw, ellipse, inner sep=2]| \sst{2\\3\\4}
    \end{tikzcd}
  \]
  In particular, there are 2 non-(co)tilting Wakamatsu tilting modules: $P(2) \oplus P(3) \oplus P(4) \oplus X$ with $X = \sst{1\\2}, \sst{3\\4}$.

  In addition, this gives an example of $T_1, T_2 \in \wtilt\Lambda$ such that $T_1 \geq T_2$ but $T_1^\perp \not\supseteq T_2^\perp$. Consider $T_i = P(2) \oplus P(3) \oplus P(4) \oplus X_i$ for $i=1,2$ with $X_1 = \sst{1\\2}$ and $X_2 = \sst{1}$. The Hasse quiver above shows that  $T_1 \geq T_2$. However, we can see that $\sst{3} \in T_2^\perp$ but $\sst{3} \not\in T_1^\perp$ since $\dim_k \Ext^1_\Lambda(\sst{1\\2}, 3) = 1$, so $T_1^\perp \not\supseteq T_2^\perp$.
\end{example}

\begin{example}
  A Nakayama algebra $\Lambda$ with the second-smallest dimension admitting non-(co)tilting Wakamatsu tilting modules has Kupisch series [5, 6, 6], which is given by the quiver
  \[
    \begin{tikzcd}[sep=tiny]
      1 \ar[rr, "a"] & & 2 \ar[dl, "b"] \\
      & 3 \ar[ul, "c"]
    \end{tikzcd}
  \]
  with relations $abcab = bcabca = 0$. As in the preceding example, since there are 2 indecomposable projective-injective $\Lambda$-modules $P(2)$ and $P(3)$ and $|\Lambda| = 3$, Wakamatsu tilting $\Lambda$-modules are in bijective correspondence with indecomposable non-projective-injective self-orthogonal $\Lambda$-modules $X$.
  There are 8 such $X$, and $(\wtilt\Lambda, \leq)$ forms a poset. Its Hasse quiver is given below, where we only show $X$.
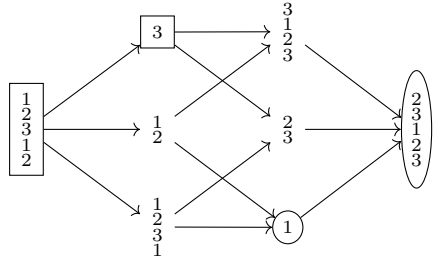
  \[
    \begin{tikzcd}[row sep=0, column sep=large]
      & |[draw, rectangle]| \sst{3} \ar[rd] \rar
      & \sst{3\\1\\2\\3} \ar[rd]\\
      |[draw, rectangle]| \sst{1\\2\\3\\1\\2} \ar[ru] \rar \ar[rd]
      & \sst{1\\2} \ar[ru] \ar[rd]
      & \sst{2\\3} \rar
      & |[draw, ellipse, inner sep = 2]| \sst{2\\3\\1\\2\\3}\\
      & \sst{1\\2\\3\\1} \rar \ar[ru]
      & |[draw, ellipse, inner sep = 2]| \sst{1} \ar[ru]
    \end{tikzcd}
  \]
  In particular, there are 4 non-(co)tilting Wakamatsu tilting modules:
\end{example}

\begin{example}
  The following algebra is taken from \cite[Example 3.1]{wakamatsu}. Consider the algebra $\Lambda$ given by the quiver
  \[
    \begin{tikzcd}
      1 \rar[shift left] & 2 \lar[shift left] \rar & 3 \rar & 4 \rar[shift left] & 5 \lar[shift left]
    \end{tikzcd}
  \]
  with relation $\rad^2 = 0$. There are 5 Wakamatsu tilting modules, and $(\wtilt\Lambda, \leq)$ is totally ordered, with the Hasse quiver indicated below. Since $P(1)$ and $P(4)$ are projective-injective, they are included in all Wakamatsu tilting modules. Hence, we only show the other summands.
  \[
    \begin{tikzcd}
      |[draw, rectangle]| \sst{2\\1 \, 3} \oplus \sst{3\\4} \oplus \sst{5\\4} \rar &
      \sst{2\\1 \, 3} \oplus \sst{3\,5\\4} \oplus \sst{5\\4} \rar &
      \sst{2\\1 \, 3} \oplus \sst{3} \oplus \sst{3\,5\\4} \rar &
      \sst{2\\1} \oplus \sst{2\\1 \, 3} \oplus \sst{3\,5\\4} \rar &
      |[draw, ellipse, inner sep = 2]| \sst{2\\1} \oplus \sst{2\\3} \oplus \sst{3\,5\\4}
    \end{tikzcd}
  \]
\end{example}

The examples we have seen suggest that, if there is a Hasse arrow $T_1 \to T_2$ in $(\wtilt\Lambda, \leq)$, then only one indecomposable direct summand differs between $T_1$ and $T_2$. This is known to be the case for tilting theory \cite[Theorem 2.1]{HU-partial}. However, our next two examples will show that this does not apply to Wakamatsu tilting modules.
\begin{example}
  Let $\Lambda$ be a Nakayama algebra with Kupisch series [3, 4, 3, 4], which is given by the quiver
  \[
    \begin{tikzcd}
      1 \rar["a"] & 2 \dar["b"] \\
      4 \uar["d"] & 3 \lar["c"]
    \end{tikzcd}
  \]
  with relations $abc = cda = 0$. There are 2 indecomposable projective-injective modules $P(2)$ and $P(4)$, and the following is the Hasse quiver of $(\wtilt\Lambda, \leq)$, where we only show the remaining two direct summands.
  
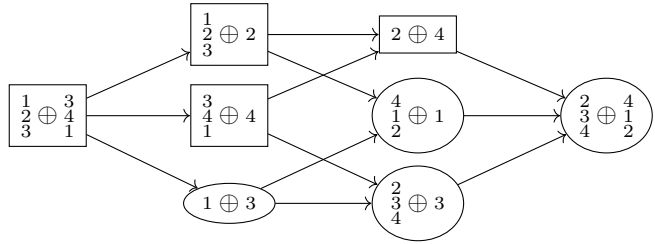
\[
    \begin{tikzcd}[row sep=tiny, column sep=large]
      & |[draw, rectangle]| \sst{1\\2\\3} \oplus \sst{2} \ar[rd] \rar
      & |[draw, rectangle]| \sst{2} \oplus \sst{4} \ar[rd]\\
      |[draw, rectangle]| \sst{1\\2\\3} \oplus \sst{3\\4\\1} \ar[ru] \rar \ar[rd]
      &|[draw, rectangle]| \sst{3\\4\\1} \oplus \sst{4} \ar[ru] \ar[rd]
      & |[draw, ellipse, inner sep = 2]| \sst{4\\1\\2} \oplus \sst{1} \rar
      & |[draw, ellipse, inner sep = 2]| \sst{2\\3\\4} \oplus \sst{4\\1\\2}\\
      & |[draw, ellipse, inner sep = 2]| \sst{1} \oplus \sst{3} \rar \ar[ru]
      & |[draw, ellipse, inner sep = 2]| \sst{2\\3\\4} \oplus \sst{3} \ar[ru]
    \end{tikzcd}
  \]
  From this diagram, we can observe that for the 4 arrows connecting tilting modules to cotilting modules, the two modules differ by two indecomposable summands.
\end{example}

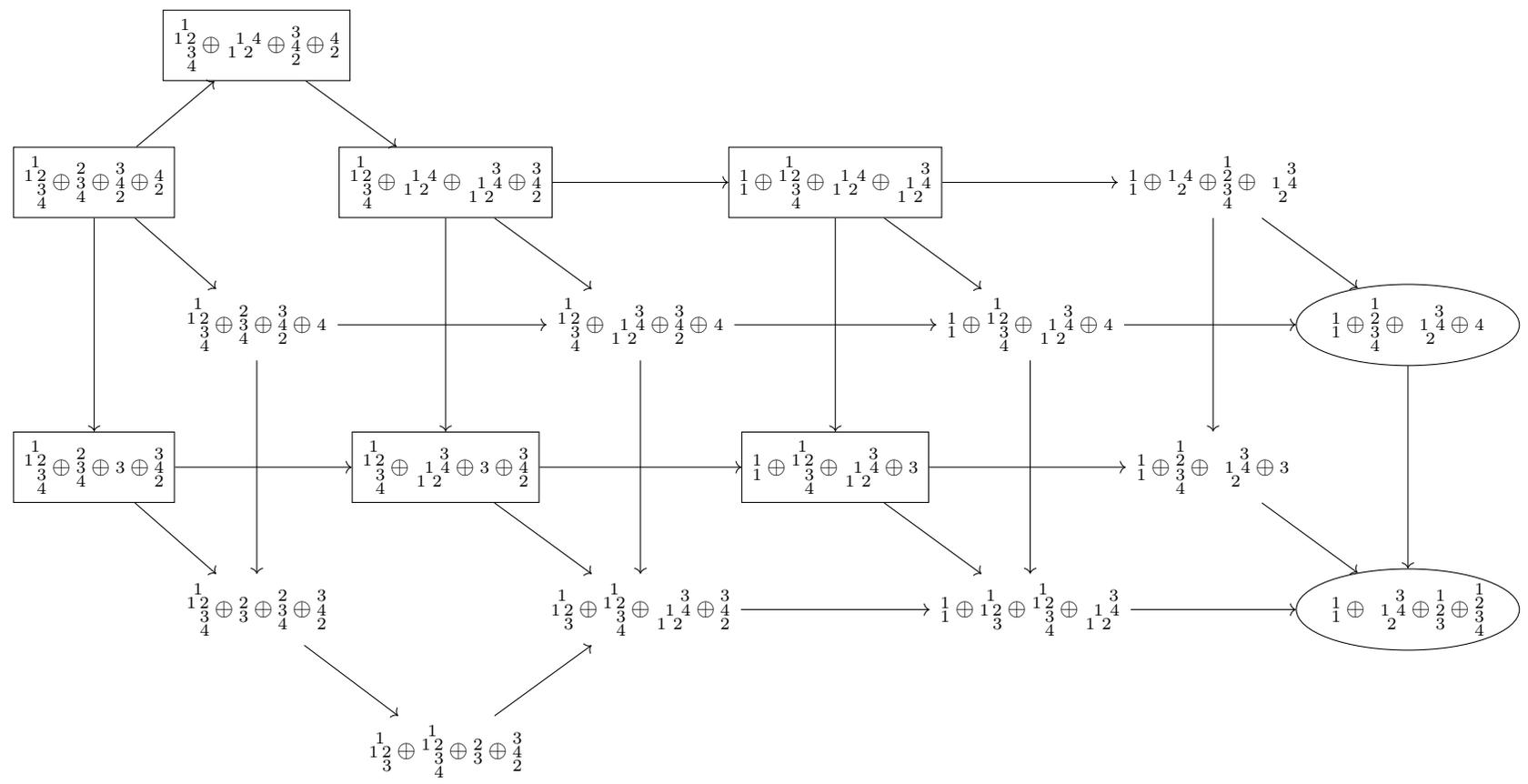
\begin{sidewaysfigure}[p]
  \vspace{15cm}
  \centering
  \begin{tikzcd}[column sep=-5, row sep=large]
    &
    |[draw, rectangle]|
    \sst{1\\ 1\, 2\\\,\,\,\, 3\\\,\,\,\, 4} \oplus
    \sst{\,\,\, 1 \,\, 4\\\!\!1\,\, 2}
    \oplus \sst{3\\4\\2} \oplus \sst{4\\2}
    \ar[rd]
    \\
    |[draw, rectangle]|
    \sst{1\\ 1\, 2\\\,\,\,\, 3\\\,\,\,\, 4} \oplus
    \sst{2\\3\\4} \oplus \sst{3\\4\\2} \oplus \sst{4\\2}
    \ar[ru] \ar[rd] \ar[dd] & &
    |[draw, rectangle]|
    \sst{1\\ 1\, 2\\\,\,\,\, 3\\\,\,\,\, 4} \oplus
    \sst{\,\,\, 1 \,\, 4\\\!\!1\,\, 2} \oplus
    \sst{\quad 3 \\ \,\,\,1 \,\, 4 \\ \!\! 1 \,\,2}
    \oplus \sst{3\\4\\2}
    \ar[rr] \ar[rd] \ar[dd] & &
    |[draw, rectangle]|
    \sst{1\\1} \oplus
    \sst{1\\ 1\, 2\\\,\,\,\, 3\\\,\,\,\, 4} \oplus
    \sst{\,\,\, 1 \,\, 4\\\!\!1\,\, 2} \oplus
    \sst{\quad 3 \\ \,\,\,1 \,\, 4 \\ \!\! 1 \,\,2}
    \ar[rr] \ar[rd] \ar[dd] & &
    \sst{1\\1} \oplus
    \sst{1\,\,4 \\ \,2} \oplus
    \sst{1\\2\\3\\4} \oplus
    \sst{\quad 3 \\ \,\,\,1 \,\, 4 \\ \,\,2}
    \ar[rd] \ar[dd]
    \\
    &
    \sst{1\\ 1\, 2\\\,\,\,\, 3\\\,\,\,\, 4} \oplus
    \sst{2\\3\\4} \oplus \sst{3\\4\\2} \oplus \sst{4}
    \ar[rr] \ar[dd] & &
    \sst{1\\ 1\, 2\\\,\,\,\, 3\\\,\,\,\, 4} \oplus
    \sst{\quad 3 \\ \,\,\,1 \,\, 4 \\ \!\! 1 \,\,2} \oplus
    \sst{3\\4\\2} \oplus
    \sst{4}
    \ar[rr] \ar[dd] & &
    \sst{1\\1} \oplus
    \sst{1\\ 1\, 2\\\,\,\,\, 3\\\,\,\,\, 4} \oplus
    \sst{\quad 3 \\ \,\,\,1 \,\, 4 \\ \!\! 1 \,\,2} \oplus
    \sst{4}
    \ar[rr] \ar[dd] & &
    |[draw, ellipse, inner sep=1]|
    \sst{1\\1} \oplus
    \sst{1\\2\\3\\4} \oplus
    \sst{\quad 3 \\ \,\,\,1 \,\, 4 \\ \,\,2} \oplus
    \sst{4}
    \ar[dd]
    \\
    |[draw, rectangle]|
    \sst{1\\ 1\, 2\\\,\,\,\, 3\\\,\,\,\, 4} \oplus
    \sst{2\\3\\4} \oplus \sst{3} \oplus \sst{3\\4\\2}
    \ar[rr] \ar[rd] & &
    |[draw, rectangle]|
    \sst{1\\ 1\, 2\\\,\,\,\, 3\\\,\,\,\, 4} \oplus
    \sst{\quad 3 \\ \,\,\,1 \,\, 4 \\ \!\! 1 \,\,2} \oplus
    \sst{3} \oplus
    \sst{3\\4\\2}
    \ar[rr] \ar[rd] & &
    |[draw, rectangle]|
    \sst{1\\1} \oplus
    \sst{1\\ 1\, 2\\\,\,\,\, 3\\\,\,\,\, 4} \oplus
    \sst{\quad 3 \\ \,\,\,1 \,\, 4 \\ \!\! 1 \,\,2} \oplus
    \sst{3}
    \ar[rr] \ar[rd] & &
    \sst{1\\1} \oplus
    \sst{1\\2\\3\\4} \oplus
    \sst{\quad 3 \\ \,\,\,1 \,\, 4 \\ \,\,2} \oplus
    \sst{3}
    \ar[rd]
    \\
    &
    \sst{1\\ 1\, 2\\\,\,\,\, 3\\\,\,\,\, 4} \oplus
    \sst{2\\3} \oplus
    \sst{2\\3\\4} \oplus \sst{3\\4\\2}
    \ar[rd] & &
    \sst{1\\ 1\, 2\\\,\,\,\, 3} \oplus
    \sst{1\\ 1\, 2\\\,\,\,\, 3\\\,\,\,\, 4} \oplus
    \sst{\quad 3 \\ \,\,\,1 \,\, 4 \\ \!\! 1 \,\,2} \oplus
    \sst{3\\4\\2}
    \ar[rr] & &
    \sst{1\\1} \oplus
    \sst{1\\ 1\, 2\\\,\,\,\, 3} \oplus
    \sst{1\\ 1\, 2\\\,\,\,\, 3\\\,\,\,\, 4} \oplus
    \sst{\quad 3 \\ \,\,\,1 \,\, 4 \\ \!\! 1 \,\,2}
    \ar[rr] & &
    |[draw, ellipse, inner sep=1]|
    \sst{1\\1} \oplus
    \sst{\quad 3 \\ \,\,\,1 \,\, 4 \\ \,\,2} \oplus
    \sst{1\\2\\3} \oplus
    \sst{1\\2\\3\\4}
    \\
    & &
    \sst{1\\ 1\, 2\\\,\,\,\, 3} \oplus
    \sst{1\\ 1\, 2\\\,\,\,\, 3\\\,\,\,\, 4} \oplus
    \sst{2\\3} \oplus
    \sst{3\\4\\2}
    \ar[ru]
  \end{tikzcd}
  \caption{The Hasse quiver of $(\wtilt\Lambda, \leq)$ for the algebra $\Lambda$ in Example \ref{ex:large}}
  \label{fig:large}
\end{sidewaysfigure}

\begin{example}\label{ex:large}
  Let $\Lambda$ be an algebra given by the quiver
  \[
    \begin{tikzcd}
      1 \ar[loop below, "a"] \rar["b"] & 2 \rar["c"] & 3 \dar["d"] \\
      & & 4 \ar[lu, "e"]
    \end{tikzcd}
  \]
  with relations $a^2 = ab = cde = ec = 0$. This is a representation-finite algebra with 18 Wakamatsu tilting modules, of which 7 are tilting and 2 are cotilting.
  We can check that $(\wtilt\Lambda, \leq)$ is a poset, and its Hasse quiver is given by Figure \ref{fig:large}. As in the previous example, we can observe that for some Hasse arrows, the two modules differ at more than one direct summands.

\end{example}

Now the following gives an example of $\Lambda$ such that $(\wtilt\Lambda, \leq)$ is not a poset.
\begin{example}\label{ex:not-poset}
  Consider the Nakayama algebra $\Lambda$ with Kupisch series [5, 6, 5, 6, 6, 6], which is given by the quiver
  \[
    \begin{tikzcd}
      1 \rar["a"] & 2 \rar["b"] & 3 \dar["c"] \\
      6 \uar["f"] & 5 \lar["e"] & 4 \lar["d"]
    \end{tikzcd}
  \]
  with relations $abcde = cdefa = defabc = efabcd = 0$. One can verify that there are 36 Wakamatsu tilting modules, of which 4 are tilting and 4 are cotilting. We shall now show that $(\wtilt\Lambda, \leq)$ is not a poset.
  Since there are 4 indecomposable projective-injective $\Lambda$-modules, $P(2), P(4), P(5), P(6)$, it is enough to give the remaining two indecomposable summands to represent each Wakamatsu tilting $\Lambda$-module. Then we can check that the following relations hold in $(\wtilt\Lambda, \leq)$:
  \[
    \sst{3} \oplus \sst{5\\6} < \sst{3} \oplus \sst{6} < \sst{3\\4} \oplus \sst{6}.
  \]
  However, $\sst{3} \oplus \sst{5\\6} \not< \sst{3\\4} \oplus \sst{6}$, because $\dim_k \Ext_\Lambda^1(\sst{3\\4} \oplus \sst{6}, \sst{3} \oplus \sst{5\\6}) = 1$.

  Although we omit the calculation here, other Nakayama algebras with the following Kupisch series also satisfy that $(\wtilt\Lambda, \leq)$ is not a poset: for rank 5, [9, 10, 9, 10, 10], [14, 15, 14, 15, 15], [19, 20, 19, 20, 20], etc.; for rank 6, [5, 6, 5, 6, 6, 6] (this example), [5, 6, 6, 5, 6, 6], [10, 12, 11, 12, 12, 11], [10, 12, 12, 11, 12, 11], [11, 11, 12, 11, 12, 12], etc.
\end{example}

We end with the following last example, borrowed from \cite[Section 2]{RS}, which shows that some conditions in Theorem \ref{thm:main2} are not equivalent for the representation-infinite case.
\begin{example}\label{ex:mso-not-wak}
  Let $\Lambda$ be an algebra given by the quiver
  \[
    \begin{tikzcd}[row sep=0]
      & 2 \ar[dd, "b", shift left] \ar[dd, "c"', shift right] \\
      1 \ar[ru, "a"]  \\
      & 3 \ar[lu, "d"]
    \end{tikzcd}
  \]
  with relations $ab = cd = da = 0$. Consider the simple module $S(1)$ corresponding to the vertex $1$. Then $S(1)$ is self-orthogonal and $\pd S(1) = 2$. In \cite[Section 2]{RS}, it is shown that $S(1)$ is maximal self-orthogonal.
  However, it is obvious that $S(1)$ is not Wakamatsu tilting.
\end{example}

\addtocontents{toc}{\SkipTocEntry}
\subsection*{Acknowledgement}
The author expresses gratitude to Osamu Iyama for offering thoughtful comments.
This work is supported by JSPS KAKENHI Grant Number JP21J00299.

\end{document}